\documentclass[letterpaper, 11pt,  reqno]{amsart}

\usepackage[margin=1.2in,marginparwidth=1.5cm, marginparsep=0.5cm]{geometry}

\usepackage{amsmath, mathrsfs, amssymb,amscd,amsthm,amsxtra, esint}
\usepackage{tikz} 
\usepackage{color}
\usepackage{xcolor}

\usepackage[implicit=true]{hyperref}

\usepackage{cases}%%%

\usepackage{upgreek}

\allowdisplaybreaks[2]

\definecolor{gr}{rgb}   {0.,   0.69,   0.23 }
\definecolor{bl}{rgb}   {0.,   0.5,   1. }
\definecolor{mg}{rgb}   {0.85,  0.,    0.85}
\definecolor{yl}{rgb}   {0.8,  0.7,   0.}
\definecolor{or}{rgb}  {0.7,0.2,0.2}

\usetikzlibrary{shapes.misc}
\usetikzlibrary{shapes.symbols}
\usetikzlibrary{shapes.geometric}

\tikzset{
	ddot/.style={circle,fill=white,draw=black,inner sep=0pt,minimum size=0.8mm},
	>=stealth,
	}

\tikzset{
	ddot2/.style={circle,fill=black,draw=black,inner sep=0pt,minimum size=0.8mm},
	>=stealth,
	}

\newtheorem{theorem}{Theorem} [section]

\newtheorem{lemma}[theorem]{Lemma}
\newtheorem{proposition}[theorem]{Proposition}
\newtheorem{remark}[theorem]{Remark}

\newtheorem{definition}[theorem]{Definition}

\DeclareMathOperator*{\intt}{\int}

\DeclareMathOperator*{\supp}{supp}
\DeclareMathOperator{\med}{med}

\DeclareMathOperator{\HS}{HS}

\DeclareMathOperator{\Id}{Id}

\DeclareMathOperator{\Ker}{Ker}

\newcommand{\1}{\hspace{0.5mm}\textup{I}\hspace{0.5mm}}
\newcommand{\I}{\mathcal{I}}

\newcommand{\noi}{\noindent}
\newcommand{\Z}{\mathbb{Z}}
\newcommand{\R}{\mathbb{R}}

\newcommand{\T}{\mathbb{T}}
\newcommand{\bul}{\bullet}

\let\Re=\undefined\DeclareMathOperator*{\Re}{Re}
\let\Im=\undefined\DeclareMathOperator*{\Im}{Im}

\newcommand{\PP}{\mathbb{P}}

\newcommand{\E}{\mathbb{E}}

\renewcommand{\H}{\mathcal{H}}

\newcommand{\CC}{\mathcal{C}}
\newcommand{\D}{\mathcal{D}}

\renewcommand{\L}{\mathcal{L}}

\newcommand{\al}{\alpha}
\newcommand{\be}{\beta}
\newcommand{\dl}{\delta}
\newcommand{\updl}{\updelta}

\newcommand{\nb}{\nabla}

\newcommand{\Dl}{\Delta}
\newcommand{\eps}{\varepsilon}

\newcommand{\g}{\gamma}
\newcommand{\G}{\Gamma}

\newcommand{\Ld}{\Lambda}
\newcommand{\s}{\sigma}
\newcommand{\Si}{\Sigma}
\newcommand{\ft}{\widehat}
\newcommand{\Ft}{{\mathcal{F}}}
\newcommand{\wt}{\widetilde}
\newcommand{\cj}{\overline}
\newcommand{\dx}{\partial_x}
\newcommand{\dt}{\partial_t}
\newcommand{\dd}{\partial}

\newcommand{\ta}{\theta}

\renewcommand{\l}{\ell}
\renewcommand{\o}{\omega}

\newcommand{\les}{\lesssim}
\newcommand{\ges}{\gtrsim}

\newcommand{\jb}[1]
{\langle #1 \rangle}

\newcommand{\ind}{\mathbf 1}

\renewcommand{\S}{\mathcal{S}}

\newcommand{\N}{\mathbb{N}}

\newcommand{\uu}{\mathbf{u}}
\newcommand{\zz}{\mathbf{z}}

\newcommand{\z}{\zeta}

\newcommand{\too}{\longrightarrow}

\usepackage[most]{tcolorbox}

\newcommand{\Hs}{\mathscr{H}}

\newcommand{\HH}{\mathcal{H}}

\newcommand{\Sym}{\textup{\texttt{Sym}}}
\newcommand{\Om}{\Omega}

\newcommand{\LOP}{\mathcal{L}}

\newcommand{\XX}{\mathbf{X}}
\newcommand{\YY}{\mathbf{Y}}

\newcommand{\hf}{\mathfrak{h}}
\newcommand{\ff}{\mathfrak{g}}

\newcommand{\Nb}{{\bf N}}
\newcommand{\ab}{{\bf a}}
\newcommand{\A}{\mathcal{A}}

\newtheorem*{ackno}{Acknowledgements}

\numberwithin{equation}{section}
\numberwithin{theorem}{section}

\makeatletter
\@namedef{subjclassname@2020}{\textup{2020} Mathematics Subject Classification}
\makeatother

\begin{document}
\baselineskip = 14pt

\title[1-$d$ stochastic heat equation]
{A remark on pathwise well-posedness of \\
the 1-$\pmb d$ stochastic heat equation}

\author[Y.~Shao, J.~Li, and T.~Oh]
{Yufei Shao, Jiawei Li,
and  Tadahiro Oh}

\address{Yufei Shao, 
School of Mathematics and Statistics, Beijing Institute of Technology, Beijing 100081, China}

\email{yufeishao@bit.edu.cn}

\address{Jiawei Li, School of Mathematics\\
The University of Edinburgh\\
and The Maxwell Institute for the Mathematical Sciences\\
James Clerk Maxwell Building\\
The King's Buildings\\
Peter Guthrie Tait Road\\
Edinburgh\\ 
EH9 3FD\\
 United Kingdom}

\email{jiawei.li@ed.ac.uk}

%%
%\address{
%Tadahiro Oh, 
%School of Mathematics and Statistics, Beijing Institute of Technology, Beijing 100081, China\\
%and
%School of Mathematics\\
%The University of Edinburgh\\
%and The Maxwell Institute for the Mathematical Sciences\\
%James Clerk Maxwell Building\\
%The King's Buildings\\
%Peter Guthrie Tait Road\\
%Edinburgh\\
%EH9 3FD\\
% United Kingdom}

%
\address{
Tadahiro Oh, 
School of Mathematics\\
The University of Edinburgh\\
and The Maxwell Institute for the Mathematical Sciences\\
James Clerk Maxwell Building\\
The King's Buildings\\
Peter Guthrie Tait Road\\
Edinburgh\\
EH9 3FD\\
 United Kingdom,
and  School of Mathematics and Statistics, Beijing Institute of Technology, Beijing 100081, China}

\email{hiro.oh@ed.ac.uk}

% 65M12}


\begin{abstract}

We study pathwise well-posedness
of the stochastic heat equation (SHE) 
with a multiplicative noise on the circle.
By combining
the convolution Young and rough integration theory,  
introduced by Gubinelli and Tindel (2010), 
with  the random tensor estimate approach
to pathwise well-posedness 
of stochastic dispersive PDEs with multiplicative noises,  introduced by 
Chapouto and  the second  and third authors (2026), 
we establish pathwise well-posedness of SHE in both the Young
and rough cases, improving the results in 
Gubinelli and Tindel (2010).
In particular, in the rough case (= the white-in-time case), 
our result covers the case of almost space-time white noise, 
thus establishing an optimal result
within the framework of one-parameter rough paths.

\end{abstract}

%\date{\today}
%%
%
\maketitle

\tableofcontents

\newpage

\section{Stochastic heat equation}\label{SEC:1}

%\section{Introduction}\label{SEC:1}
%\subsection{Stochastic heat equation}

We consider the following stochastic heat equation (SHE) 
with a multiplicative noise, posed on the circle $\T = \R/(2\pi \Z)$:\footnote{By convention, we endow
$\T$ with the normalized Lebesgue measure $ dx_{\T} =  (2\pi)^{-1}dx$
such that we do not need to carry factors involving $2\pi$.}
\begin{align}
\begin{cases}
\dt u +(1- \dx^2) u = u \phi \zeta\\
u|_{t = 0} = u_0.
\end{cases}
\label{NLH1}
\end{align}

\noi
Here, 
$\phi$ is a Hilbert-Schmidt operator from $L^2(\T)$ to 
$H^\s(\T)$ 
for some $\s \in \R$  
(such that\footnote{More precisely, $\phi W^\be(t)$ has spatial regularity $\s$, 
where $W^\be$ is as in \eqref{W1}.}
 the noise $\phi \zeta$ has spatial regularity $\s$)
and $\zeta$ denotes a fractional-in-time\,/\,white-in-time and white-in-space noise.
Heuristically, one may think of $\zeta$ as 
\begin{align}
\text{``}\,\z = \jb{\dt}^{-\al}\xi\,\text{''}
\label{W0a}
\end{align}
for some $0 \le \al <  \frac 12$, where 
$\jb{\,\cdot\,} = (1+ |\cdot|^2)^\frac12$ and 
$\xi$ denotes 
a (Gaussian) space-time white noise on $\R_+ \times \T$
whose space-time covariance is (formally) given by 
\begin{align*}
 \E[ \xi(t_1, x_1)\xi(t_2, x_2) ] = \dl(t_1 - t_2) \dl (x_1 - x_2)
\end{align*} 

\noi
for $t_1, t_2 \in \R_+$  and $x_1, x_2 \in \T$
with $\dl$ denoting the Dirac delta function.
See \eqref{ze1}
for the precise meaning of $\z$.

We say that $u$ is a solution to \eqref{NLH1}
with initial data $u|_{t = 0} = u_0$
if $u$ satisfies the following
 Duhamel formulation (= mild formulation):
\begin{align}
u(t) = S(t) u_0 + \Psi(u)(t), 
\label{NLH2}
\end{align}

\noi
where $S(t) = e^{t(\dx^2 - 1)}$ denotes the heat semigroup and $\Psi(u)$ denotes the stochastic convolution, 
representing  the effect of the multiplicative noise, given by 
\begin{align}
\Psi(u)(t) = 
\int_0^t S(t-t')\big[ u(t') \phi dW^\be (t')\big].
\label{psi1}
\end{align}

\noi
Here, 
$W^\be$ denotes the cylindrical fractional Wiener process on $L^2(\T)$ given by 
\begin{align}
W^\be(t) 
 = \sum_{n\in \Z}  
B_n(t) e_n, 
\label{W1}
\end{align}

 \noi
where $e_n(x) = e^{in x}$
and $\{B_n \}_{n \in \Z}$ is a family of independent 
complex-valued 
fractional Brownian motions with Hurst parameter\footnote{Note that $\al$ in \eqref{W0a}
is given by $\al = \be -\frac 12$.} $\frac 12 \le \be < 1$, 
conditioned that 
\begin{align}
B_{-n} = \cj{B_n}, \quad n \in \Z.
\label{W2}
\end{align}

\noi
Namely, we have
\begin{align}
\z = \dt W^\be.
\label{ze1}
\end{align}

\noi
When $\be = \frac 12$, 
the sequence $\{B_n \}_{n \in \Z}$ 
reduces to a family of independent 
standard
complex-valued 
 Brownian motions, satisfying \eqref{W2}, 
 and thus $W^{ \frac 12}$ corresponds 
 to an $L^2$-cylindrical Wiener process.
See Subsection \ref{SUBSEC:FBM}
for a review on fractional Brownian motions and 
stochastic integrals with respect to them.

%Ito solution theory (see \cite{Walsh})

Our main goal in this paper
is to revisit {\it pathwise} well-posedness issues
for SHE \eqref{NLH1}, 
previously studied in \cite{GLT, GT10, HP, GH}.
In \cite{GT10}, 
Gubinelli and Tindel
generalized Lyons' rough path theory \cite{Lyons1}
and developed an algebraic integration theory
adapted to an analytic semigroup.
More precisely, they extended 
 the sewing lemma 
and  controlled paths, introduced in a seminal work \cite{Gub04} by Gubinelli, 
to  the current convolution setting.\footnote{There are various names
attached to the approach in \cite{GT10}.
See, for example, \cite[Exercise 4.16]{FH20}.
Following~\cite{GT10}, we use the term ``convolution'' 
%(as in the ``convolution sewing map'' in \cite[Theorem 3.5]{GT10})
since relevant operators
(such as $\XX^1$ in \eqref{sto1x}) have a convolution-in-time structure.
}
In particular, 
by constructing
the stochastic convolution $\Psi(u)$ in \eqref{psi1}
as a convolution Young\,/\,rough integral, 
they proved well-posedness of SHE \eqref{NLH1} 
with  $\phi = \jb{\dx}^{-\nu}$.
See Subsections \ref{SUBSEC:Y1}
and \ref{SUBSEC:R1}
below for a review on the algebraic part of their approach.

In this paper, 
we will employ the algebraic structure introduced in \cite{GT10}
but introduce an improvement on the {\it analytical side}
by incorporating the random tensor estimate, 
originally introduced in the context 
of random variables by Deng, Nahmod, and Yue \cite{DNY3}, 
to the current stochastic setting.
In a series of recent works \cite{CLO2, COZ, CGLLO2, OShao, CLOO}, 
 the last two authors with A.\,Chapouto
developed
pathwise well-posedness theory 
of stochastic dispersive PDEs with multiplicative noises, 
which had been open for  more than several decades
since the inception of Ito solution theory (and random field solution theory)
in 
\cite{Walsh}
for wave equations
and 
in~\cite{DD1, DD2}
for  Schr\"odinger equations.
A key ingredient in this development
was the random tensor estimate
for multiple stochastic integrals
(for both the fractional-in-time and white-in-time cases); see Lemma~\ref{LEM:RT}.
See Section \ref{SEC:RT}
for a further discussion on the random tensor estimate 
for multiple stochastic integrals.
The random tensor estimate 
allows us 
 to estimate operator norms
of relevant (random) drivers in a much more effective manner
as compared to~\cite{GT10}, 
where the operator norm was  crudely bounded by 
the Hilbert-Schmidt norm, 
since the latter is 
 more 
 amenable
to stochastic analysis
(in particular, the Wiener-Ito isometry).

\medskip

We now state our main results.
As a  simplification,  we only consider the case when 
$\phi$ is spatially homogeneous, namely 
$\phi$ is a Fourier multiplier operator
given by 
\begin{align}
\phi (e_n) = \phi_n e_n.
\label{phi1}
\end{align}

\noi
Under this assumption, we have 
\begin{align}
\| \phi \|_{\HS(L^2; H^\s)}= \|\jb{n}^\s\phi_n\|_{\l^2_n}.
\label{phi1a}
\end{align}

We first state a pathwise well-posedness result in the 
Young case (= the fractional-in-time case
with the Hurst parameter $\frac 12 < \be < 1$).

\begin{theorem}[Young case]\label{THM:1}
Let $\frac 12 < \be < 1$ and $\s \in \R$ satisfy
\begin{align}
\s  >  -   2 \be + 1.
\label{TH1}
\end{align}

\noi
Then, there exists small $s > 0$ such that, 
given $\phi \in \HS(L^2(\T); H^\s(\T))$, 
 the stochastic heat equation
\eqref{NLH1} is pathwise globally well-posed in $H^s(\T)$.

\end{theorem}

Following \cite[Section 4]{GT10}, we define the first order driver $\XX^1$ by setting
\begin{align}
\begin{split}
\XX^1_{t, r}(f)
&  =  \int_r^t
S(t- t') \big[ (S(t' - r) f)
\phi dW^\be(t') \big]
\end{split}
\label{sto1x}
\end{align}

\noi
for $t > r \ge 0$, 
 where  $f$ is a function on $\T$.
Then, we make sense of the stochastic convolution $\Psi(u)$
in \eqref{psi1} as
the convolution Young integral $\I^{\XX^1}(u)$
with the first order driver $\XX^1$;
see Subsection \ref{SUBSEC:Y1}.
This allows us to rewrite the Duhamel formulation \eqref{NLH2}
as the following
 convolution Young differential equation
(YDE):
\begin{align}
u(t) = S(t) u_0 + \I^{\XX^1}(u)(t)
\label{YDE1}
\end{align}

\noi
whose global well-posedness is established in a pathwise manner in Subsection \ref{SUBSEC:Y3}.

The main task is then to 
construct
the convolution Young integral $\I^{\XX^1}(u)$.
In doing so, 
we follow the incremental formulation as in \cite{Gub04, GT10}, 
where the  matter is reduced to 
establishing
 almost sure regularity properties
of the random driver~$\XX^1$.
In \cite{GT10}, 
Gubinelli and Tindel  estimated the operator
norm of $\XX^1$ by its Hilbert-Schmidt norm, which led to a loss.
In our novel approach, 
we study  regularity properties of $\XX^1$
by applying the random tensor estimate
for (multiple) stochastic integrals (Lemma \ref{LEM:RT}), 
developed in~\cite{OWZ, CLO2, COZ};
see Subsection~\ref{SUBSEC:Y2}.
A primary reason for working with the Hilbert-Schmidt norm in \cite{GT10}
(which may be viewed as the $\l^2$-norm for eigenvalues in some setting)
is its compatibility with the Wiener-Ito isometry
(related to the $L^2(\Om)$-norm).
On the other hand, 
the operator norm may be regarded as the $\l^\infty$-norm, which is not
compatible with the $L^2(\Om)$-norm
for the Wiener-Ito isometry.
The main idea behind the random tensor estimate, due to Deng, Nahmod, 
and Yue \cite{DNY3}, 
is as follows;
instead of working with the Hilbert-Schmidt
norm of a random operator ($\XX^1$ in our setting), 
we work with the 
Hilbert-Schmidt norm (= the $\l^2$-norm) of 
a very high power of $\XX^1$, say the $k$th power. 
This gives a control on 
 the $\l^{2k}$-norm of (the kernel of) $\XX^1$,
 providing a good approximation  
for  the $\l^\infty$-norm (= the operator norm) when $k \gg 1$, 
%by the $\l^p$-norm for large  $p \gg 1$, 
which is far more effective than approximating
the $\l^\infty$-norm by the $\l^2$-norm.
In the proof of the random tensor estimate, 
the 
kernel  of 
a high power of $\XX^1$ is estimated in an inductive manner.
We point out that 
the Wiener chaos estimate (Lemma~\ref{LEM:hyp})
plays a crucial role in reducing the matter
to estimating the 
Hilbert-Schmidt norm of  a high power of $\XX^1$
(more precisely, the $\l^2$-norm of the kernel
of a high power of $\XX^1$); 
see~\cite{OWZ, CLO2} for the proofs.
We note that 
a precursor of this idea (of working with a higher power of a random operator) already appears in a seminal work
\cite{BO96} by Bourgain.
See also Remark \ref{REM:RT}.
Lastly, we remark that, as pointed out in  
 \cite[Remark 6.6]{Bring2}, 
the random tensor estimate 
is closely related to operator bounds for structured random matrices
and  can be proven, using 
the 
non-commutative 
Khintchine  inequality 
\cite[Theorem~3.2]{vH}.
%\cite[Corollary~3.3]{vH}.
%
%the method in \cite{vH}, 
%in particular 
See \cite{Bring2, Kaneshiro}
for a further discussion.
%
%{\bf \Rd can be viewed 
% as an iteration of the 
%non-commutative Khinchin inequality.
%}
%

Once we establish almost sure regularity properties of the driver $\XX^1$
(Proposition \ref{PROP:drive2}), 
we can simply 
 apply
the convolution sewing lemma (Lemma~\ref{LEM:sew2})
to construct the convolution Young integral $\I^{\XX^1}(u)$
in a pathwise manner.
We note that local well-posedness of the convolution YDE~\eqref{YDE1}
then follows from a standard contraction
argument, whereas its global well-posedness
follows from an iterative application 
of the local well-posedness  argument.

\begin{remark}\rm

In \cite[Theorem 4.12]{GT10}, 
Gubinelli and Tindel studied the case $\phi = \jb{\dx}^{-\nu}$, $0 \le \nu < 1$, 
which belongs to $\HS(L^2(\T); H^\s(\T))$
if and only if  $\s < \nu - \frac 12$.
In this setting, the condition~\eqref{TH1} in Theorem \ref{THM:1} reduces to 
\begin{align}
\nu  >  -   2 \be + \frac 32.
\label{YS0a3}
\end{align}

\noi
Compare this with the condition 
 in 
\cite[Remark 4.9]{GT10}:
\begin{align}
\nu > - 4\be +3.
\label{YS0a4}
\end{align}

\noi
Recalling that 
$\nu \ge 0$, 
our condition \eqref{YS0a3} provides an improvement
over the condition~\eqref{YS0a4}
from \cite{GT10}
(restricted to $\frac 12 < \be < \frac 34$; 
for $\frac 34 \le  \be < 1$, 
both \eqref{YS0a3}
and \eqref{YS0a4} provide the same conclusion  under 
the restriction $\nu \ge 0$).

\medskip

\end{remark}

Next, we  state a pathwise well-posedness result 
in the rough case (= 
the white-in-time case with the Hurst parameter $\be = \frac 12$)
in a somewhat formal manner.

\begin{theorem}[rough case]\label{THM:2}
Let $\be = \frac 12$ and $\s > - \frac 12$.
Then, there exists small $s > 0$ such that, given $\phi \in \HS(L^2(\T); H^\s(\T))$, 
 the stochastic heat equation
\eqref{NLH1} is pathwise  globally well-posed in $H^s(\T)$.

\end{theorem}

Following \cite[Section 5]{GT10}, 
we introduce the second order driver $\XX^2$ and 
the third order driver $\XX^3$ 
via the recursive relation \eqref{high1a} (see also \eqref{high1});
see \eqref{high2} and \eqref{high3}
for direct definitions of $\XX^2$ and $\XX^3$, respectively.
Then, we make sense of the stochastic convolution $\Psi(u)$
in~\eqref{psi1} as
the convolution rough integral $\I^{\cj \XX}(u)$
with the  driver $\cj \XX = (\XX^1, \XX^2, \XX^3)$
by imposing a second order controlled structure on $u$;
see Subsection \ref{SUBSEC:R1}.
This allows us to rewrite the Duhamel formulation \eqref{NLH2}
as the following
 convolution rough differential equation
(RDE):
\begin{align*}
u(t) = S(t) u_0 + \I^{\cj \XX}(u)(t).
%\label{YDE2}
\end{align*}

\noi
Once the convolution rough integral $\I^{\cj \XX}(u)$
is constructed in a suitable manner, 
a standard contraction with the second order controlled structure on the unknown $u$
yields
local and global well-posedness  in a pathwise manner;
see Subsection \ref{SUBSEC:R3}.

As in the Young case, 
the main task in constructing the 
convolution rough integral $\I^{\cj \XX}(u)$ 
with the driver 
$\cj \XX = (\XX^1, \XX^2, \XX^3)$
is to establish almost sure regularity properties
of the higher order drivers $\XX^2$ and $\XX^3$.
We achieve this by applying the random tensor
estimate 
and arguing inductively 
(see \eqref{ZYS3} and \eqref{ZZYS3})
by exploiting 
the 
smoothing property of the heat semigroup $S(t)$
(on the Fourier side).
See Subsection \ref{SUBSEC:R2}.

\begin{remark}\rm
Recall that the case of a space-time white noise corresponds
to $\phi = \Id$ which belongs to 
$ \HS(L^2(\T); H^\s(\T))$ for $\s < - \frac 12$.
Theorem \ref{THM:2} provides the condition 
$\s > - \frac 12$, thus allowing us
to treat 
an almost space-time white noise, 
for example,  $\phi = \jb{\dx}^{-\nu}$ for any $\nu > 0$.
This yields a significant improvement
over \cite{GT10}, 
where the condition $\nu > \frac 13$
was obtained 
via the same third order expansion as explained above.

We also note that Theorem \ref{THM:2} is optimal in the following sense.
When $\phi = \Id$ (with $\be = \frac 12$), the noise $\z$ in \eqref{NLH1}
corresponds to the space-time white noise
$\xi$ whose  temporal and spatial regularities are $-\frac 12- \eps$, $\eps > 0$, almost surely.
In this case, the stochastic convolution $\Psi$ (with $u \equiv1$), given by 
\begin{align*}
\Psi(t) = 
\int_0^t S(t-t') \phi dW^\be (t'), 
\end{align*}

\noi
has 
  temporal and spatial regularities $\frac 12- \eps$, $\eps > 0$, almost surely.
In studying \eqref{NLH1}, 
we  need to understand the formal product $\Psi \xi$
as the first approximation.
However, this formal product 
suffers the deficiency of regularity in {\it both}
temporal and spatial directions.
Namely,  the sum of  regularities is negative in 
both temporal and spatial directions.
As such, the standard one-parameter  rough path
theory is not sufficient to treat this problem
and one needs to carry out  bi-parameter analysis
as in \cite{CG}.
We will address this problem in a forthcoming work.

Lastly, we point out that the theory of regularity structures, 
introduced by Hairer \cite{Hairer}, allows us to prove well-posedness
of \eqref{NLH1} in the case of a space-time white noise, 
since it treats temporal and spatial regularities in a unified
manner (called homogeneity), thus avoiding the necessity of bi-parameter analysis.
See  \cite{HP}.

\end{remark}

\begin{remark}\label{REM:GT2}\rm
(i) For simplicity of presentation, 
we restricted our attention to the one-dimensional case, 
but our analysis can be easily adapted to higher dimensions, 
improving the corresponding results in \cite{GT10}.
By the same reason, 
 we only considered the noise of the form $u \phi \z$ in this paper, 
but our approach can be easily adapted to treat 
 the noise of the form $f(u) \phi \z$ in the Young case
under a suitable assumption on $f(u)$.
In \cite[Section 6]{GT10}, 
 Gubinelli and Tindel studied
the noise of the form $u^2 \phi \xi$
in the rough case, 
using nonlinear rough path analysis
as in \cite{Gub12}; see also \cite{NX, OShao}.
In a forthcoming work, we will revisit this problem, 
using   the random tensor approach.

\smallskip

\noi
(ii)
In this paper, we study the first, second, and third order drivers
$\XX^1$, $\XX^2$, and $\XX^3$ 
for $\frac 12 \le \be < 1$
via the random tensor estimates;
see Propositions \ref{PROP:drive2}, \ref{PROP:drive3}, and \ref{PROP:drive4}.
As mentioned above, our argument is inductive
and thus can be easily adapted to treat a higher order driver~$\XX^j$
of an arbitrary degree, 
defined by the recursive relation \eqref{high1}.
It may be of interest to improve the regularity condition on $\phi$
in the Young case by considering a sufficiently higher order expansion.
See also Remark \ref{REM:GT1}.
We, however, do not pursue this issue further in this paper.

\end{remark}

\section{Notations and function spaces} %preliminary tools}
\label{SEC:2}

\subsection{Basic notations}

We use  $A\les B$ to denote an estimate of the form $A\leq CB$ for some constant $C>0$. We write $A\sim B$ if $A\les B$ and $B\les A$, while $A\ll B$ denotes $A\leq c B$ for some small constant $c> 0$. 
We may write  $\les_{\al}$ and $\sim_{\al}$ to 
emphasize the dependence on an external parameter $\al$.
We use $C>0$ to denote various constants, which may vary line by line, 
and we write $C_{\al}$
to emphasize 
the dependence on an external parameter $\al$.

Given $a, b \in \R$, we set
$a\vee b = \max(a, b)$ and 
$a\wedge b = \min(a, b)$.

%Let $V$ be a normed vector space.
%Given $K > 0$, we use $B_K$ to denote
%the closed ball in $V$ of radius $K > 0$ centered at the origin.

In expressing the dependence of a function $u$
on the time variable, we often use the short-hand notation
$u_t = u(t)$,  which is standard in probability theory and stochastic analysis.
Similarly, we often use the short-hand notation
$S_t$ for the heat semigroup 
$S(t) = e^{t(\Dl - 1)}$.
In considering an integral (in time) operator $\I$, 
a priori defined on functions on $\T$, 
we use the notation~$\I(u_\bul)$
to denote its (formal) action
on a space-time function $u$, 
where  $\bul$ denotes the variable of integration.
For example, by comparing \eqref{psi1} and \eqref{Y1}, we  have
\[ \Psi(u)(t) = \YY_{t, 0}(u_{\bul}),\]

\noi
where the right-hand side 
is merely a formal expression.

Given $n_{j_1}, \dots, n_{j_k} \in \Z$, 
we set 
\begin{align}
n_{j_1 \cdots j_k} = 
n_{j_1}+ \cdots+ n_{j_k}.
\label{short1}
\end{align}

\noi
For example, 
$n_{123} = n_1 + n_2 + n_3$.

We set $\Z_{\ge 0} = \N \cup\{0\}$
and use $2^{\Z_{\ge 0}}$ 
to denote the set of dyadic numbers $N\ge 1$.

\subsection{Function spaces}
\label{SUBSEC:2.2}

Given $s \in \R$, 
let $H^s(\T)$ be the $L^2$-based Sobolev space
defined by the norm:
\[ \| f\|_{H^s} = \bigg(\sum_{n \in \Z} \jb{n}^{2s} |\ft f(n)|^2\bigg)^\frac 12, \]

\noi
where $\ft f$ denotes the Fourier transform of $f$.

In studying space-time functions, 
we often use short-hand notations such as
$L^\infty_T H^s_x  = L^\infty([0, T]; H^s(\T))$, etc.~when there is no ambiguity.

Given Banach spaces $V$ and $W$, we use $\L(V; W)$
to denote the Banach space of bounded linear operators from $V$ to $W$.
When $V = W$, we simply set $\L(V) = \L(V;V)$.

Let $V$ be a Banach space and $T>0$.
For $n\in\N$, we set 
\begin{align*}
\Delta_{n, T} = 
\big\{ (t_1, \ldots, t_n) \in [0,T]^n: \ t_i > t_j 
\text{ for } i < j\big\}.
\end{align*}
We denote by $C_{n,T}V$ 
the space of continuous functions 
from $\Delta_{n,T}$ to $V$. When $n=1$,
we may write $C_T V$ for simplicity
and equip this space with the supremum norm: 
\begin{align*}
\|f\|_{C_T V} = \|f\|_{L^\infty_T V} = \sup_{0\leq t \leq T} \|f(t)\|_V.
\end{align*}

\noi
We define the coboundary operator 
$\updl: C_{n,T} V  \to C_{n+1,T} V$ 
as follows; 
given  $f\in C_{n,T} V$ 
and $(t_1,\ldots, t_{n+1}) \in \Dl_{{n+1},T}$, 
we set
\begin{align*}
(\updl f)_{t_1,\ldots , t_{n+1}} = 
\sum_{k=1}^{n+1} (-1)^{k} f_{t_1, \ldots, t_{k-1}, t_{k+1}, \ldots, t_{n+1}}.
\end{align*}

\noi
For example, for $f\in C_T V$
and 
$g\in C_{2,T} V$, 
we have 
\begin{align}
\begin{split}
(\updl f )_{t,r} &= f_t - f_r, \\
(\updl g)_{t_1,t_2,t_3} &= 
g_{t_1,t_3} - g_{t_1,t_2} - g_{t_2,t_3}
\end{split}
\label{dl1}
\end{align}

\noi
for $(t,r)\in\Dl_{2,T}$
and $(t_1,t_2,t_3) \in \Dl_{3,T}$.
As noted in \cite{GT10}, 
the sequence 
\begin{align*}
0 \too V \too C_{1, T}V \stackrel{\updl}{\too} C_{2, T}V
\stackrel{\updl}{\too} C_{3, T}V
\stackrel{\updl}{\too} \cdots
\end{align*}

\noi
is exact. 
In particular, we have 
 $\updl\circ\updl =0$ and 
if $f \in C_{n,T} V$ with $\updl f =0$, 
then there exists a $g\in C_{n-1,T}V$ 
such that $f= \updl g$; 
see, for example,  \cite[Lemma 2.1]{GT10}.

In the following, we go over a similar complex structure in the convolution setting;
see \cite[Section 3]{GT10}.
Let $\{S_t\}_{t\ge 0}$ be a semigroup of bounded operators on a separable Banach space $V$
 with $S_0 = \Id$, 
satisfying
\begin{align}
\|S_t f\|_{V} \le \|f\|_V.
\label{S1}
\end{align}

\noi
We then set 
\begin{align}
A_{t} = S_{t} - \Id.
\label{dl2}
\end{align}

\noi
We now define the convolution coboundary operator $\ft \updl$
by 
\begin{align*}
(\ft \updl f)_{t_1,\ldots , t_{n+1}} 
= ( \updl f)_{t_1,\ldots , t_{n+1}} 
- A_{t_1 - t_2}
 f_{t_2,\ldots , t_{n+1}}
\end{align*}

\noi
for $f \in C_{n, T}V$ and 
 $(t_1,\ldots, t_{n+1}) \in \Dl_{{n+1},T}$.
For example, for $f\in C_T V$
and 
$g\in C_{2,T} V$, 
we have 
\begin{align}
\begin{split}
(\ft \updl f)_{t, r}
&  = (\updl f)_{t, r} - A_{t-r} f_r
= f_t - S_{t-r} f_r, \\
(\ft \updl g)_{t_1, t_2, t_3} 
& = 
g_{t_1, t_3}
- g_{t_1, t_2}
- S_{t_1-t_2}g_{t_2, t_3}
\end{split}
\label{dl3}
\end{align}

\noi
\noi
for $(t,r)\in\Dl_{2,T}$
and $(t_1,t_2,t_3) \in \Dl_{3,T}$, 
where $\updl$ is as in \eqref{dl1}.
Compare \eqref{dl3} with~\eqref{dl1}.
Then, as in the non-convolution setting, 
the sequence 
\begin{align*}
0 \too S_{t}V \too C_{1, T}V \stackrel{\ft \updl}{\too} C_{2, T}V
\stackrel{\ft \updl}{\too} C_{3, T}V
\stackrel{\ft \updl}{\too} \cdots
\end{align*}

\noi
is exact, 
where $S_t V$ denotes the space of ``linear solutions'':
\[ S_t V = \{ S_t v: v \in V\}.\]

\noi
In particular, we have 
 $\ft \updl\circ\ft \updl =0$ and 
if $f \in C_{n,T} V$ with $\ft \updl f =0$, 
then there exists a $g\in C_{n-1,T}V$ 
such that $f= \ft \updl g$; 
see \cite[Proposition 3.1]{GT10}.

We also define an auxiliary operator $\wt \updl$,
acting on operator-valued increments $X \in C_{2, T} \L(V)$,  by setting
\begin{align}
\begin{split}
(\wt\updl  X)_{t_1, t_2, t_3}
& = (\ft \updl X)_{t_1, t_2, t_3} - X_{t_1, t_2}A_{t_2 - t_3}\\
& =  (\updl X)_{t_1, t_2, t_3} - A_{t_1 - t_2} X_{t_2, t_3}  
- X_{t_1, t_2}A_{t_2 - t_3}
\end{split}
\label{dl4}
\end{align}

\noi
for  $(t_1,t_2,t_3) \in \Dl_{3,T}$, 
where $A$ is as in \eqref{dl2}.
The operator
 $\wt \updl$ 
allows us to express
Chen's relation in the current convolution setting
in a concise manner;
see \eqref{chen3}, \eqref{chen1}, and \eqref{chen2}.
See also~\eqref{prod2} below.

Given $X \in C_{n, T}\L(V)$ and $f \in C_{m, T}V$, 
we define their  contraction
$X(f)  \in C_{n+m - 1, T}V$ by 
\begin{align}
\big(X(f)\big)_{t_1, \dots, t_{n+m-1}} = 
X_{t_1, \dots, t_n}
( f_{t_n, \dots, t_{n+m-1}})
\label{prod1}
\end{align}

\noi
for 
$(t_1, \dots, t_n)\in \Dl_{n, T}$
and
$(t_n, \dots, t_{n+m-1}) \in \Dl_{m, T}$.
We also recall the following product rule for $\ft \updl$
defined in \eqref{dl3};
given $X \in C_{2, T}\L(V)$ and $f \in  C_{1, T}V$, 
we have 
\begin{align}
\ft \updl (X(f)) = (\ft \updl X)(f) - X(\updl f)
= (\wt \updl X)(f) - X(\ft \updl f), 
\label{prod2}
\end{align}

\noi
where $\wt \updl$ is as in \eqref{dl4} and the right-hand side is understood in the sense of \eqref{prod1};
see  \cite[math display after~(89)]{GT10},
which follows as a variant of \cite[Lemma 3.2]{GT10}.

%Let $0 < \g < 1$.
%
% we denote by $C^\g_T V = C^\g([0, T]; V)$ the space of $\g$-H\"older continuous functions taking values in $V$, endowed  with the seminorm:
%\begin{align*}
%\|f\|_{C^\g_T V} = \sup_{(t,r)\in \Dl_{2,T}} 
%\frac{\|(\updl f)_{t,r}\|_V}{|t-r|^\g}.
%\end{align*}
%
%\noi
%We also define 
% $\CC^\g_T V = \CC^\g([0, T]; V)$ via the norm:
%\begin{align*}
%\| f \|_{\CC^\g_TV} = \| f \|_{L^\infty_TV} + \|f\|_{C^\g_T V}.
%%\label{Ho2a}
%\end{align*}
%
%
%\noi 
For $n = 2, 3$, we  introduce the spaces 
$C^\g_{n,T}V$, 
equipped with the following H\"older-type norms;
 for $g\in C_{2,T}V$ and $h\in C_{3,T}V$, we set
\begin{align}
\begin{split}
\| g\|_{C^\g_{2,T} V} & 
= \sup_{(t,r) \in \Dl_{2,T}} \frac{\|g_{t,r} \|_{V} }{|t-r|^{\g}}, \\
\| h\|_{C^\g_{3,T} V} & 
= \inf_{0<\al<\g} 
\sup_{(t_1,t_2,t_3) \in \Dl_{3,T}} 
\frac{ \|h_{t_1,t_2,t_3}\|_V}
{|t_1-t_2|^{\al} |t_2-t_3|^{\g-\al}}.
\end{split}
\label{Ho2}
\end{align}

\noi
Lastly, 
we  define the class $\ft C_{T}^\g V$
to be the collection of functions $f$ satisfying
\begin{align}
\|f \|_{\ft C_{T}^\g V} 
= \|\ft \updl f\|_{C^\g_{2, T} V} 
= \sup_{(t,r)\in \Dl_{2,T}} 
\frac{\|(\ft \updl f)_{t,r}\|_V}{|t-r|^\g} < \infty
\label{Ho3}
\end{align}

\noi
and define 
 $\ft \CC^\g_T V = \ft \CC^\g([0, T]; V)
  = \ft C_{T}^\g V\cap L^\infty_T V$
 via the norm:
\begin{align}
\| f \|_{\ft \CC^\g_TV} = \| f \|_{L^\infty_TV} + \|f\|_{\ft C^\g_T V}.
\label{Ho4}
\end{align}

\noi
From \eqref{dl3}, we have $f_t =  (\ft \updl f)_{t, 0}+ S_t f_0$.
Thus, from \eqref{Ho3} and  \eqref{S1}, we have 
\begin{align}
\|f\|_{L^\infty_T V} \le T^\g \|f\|_{\ft C^\g_T V} + \|f_0\|_{V}.
\label{Ho5}
\end{align}

\section{Random tensor estimate}
\label{SEC:RT}

In this section, 
 we first go over the basic definitions and properties of multiple stochastic integrals
 with respect to fractional Brownian motions.
In Subsection \ref{SUBSEC:RT},  
we state the random tensor estimate (Lemma \ref{LEM:RT})
which plays a crucial role
in establishing almost sure mapping properties
of the drivers $\XX^j$, $j = 1, 2, 3$.
See Subsections \ref{SUBSEC:Y2} and 
\ref{SUBSEC:R2}.

\subsection{Fractional Brownian motion and multiple stochastic integrals}
\label{SUBSEC:FBM}

In this subsection, we 
briefly go over 
 the basic definitions and properties  of fractional Brownian motions and 
Wiener integrals with  respect to fractional Brownian motions.
See \cite[Chapter 5]{Nualart06} for a further discussion.

\begin{definition}\rm
\label{DEF:fBM}
Let $0 < \be < 1$.
 A (real-valued) fractional Brownian motion $\{B(t)\}_{t\in \R_+}$ with Hurst parameter $\be$ is a centered Gaussian process with covariance given by
\begin{align*}
 \E[B(t_1) B(t_2)] =   \frac12 \Big( t_1^{2\be} + t_2^{2\be} - |t_1-t_2|^{2\be}\Big) .
\end{align*}

\noi
When $\be=\frac12$, this process reduces  to the standard Brownian motion. 
A complex-valued fractional Brownian motion $\{ B(t)\}_{t\in \R_+}$ with Hurst parameter $\be$ is a 
complex-valued centered Gaussian process  such that 
%its real and imaginary parts 
$\{ \sqrt 2 \Re B(t)\}_{t\in \R_+}$ and  $\{\sqrt 2\Im B(t)\}_{t\in \R_+}$
 are independent real-valued fractional Brownian motions with Hurst parameter $\be$
 such that 
 \[ \E\big[ |B(t_1) - B(t_2)|^2\big] = |t_1 - t_2|^{2\be}.\]

\noi
See also 
\cite[Section 5]{OST}
for a discussion on fractional Brownian motions.

\end{definition}

Let us first  introduce the following partition of  $\Z^d$:
\begin{align*}
\Z^d = (\Z^d)_+  \cup (\Z^d)_- \cup \{0\}^d ,
\end{align*}

\noi
where 
\begin{equation*}
(\Z^d)_+=\bigcup_{k=0}^{d-1} \Z^k\times \Z_{+}\times \{0\}^{d-k-1}
\qquad \text{and}\qquad 
(\Z^d)_- = -(\Z^d)_+.
\end{equation*}

\noi
%We also  set $(\Z^d)_{+,0} := (\Z^d)_{+} \cup \{0\}^d$.
Given $0 < \be < 1$, 
let $\{B_n\}_{n\in \Z^d}$ be a family of 
mutually independent complex-valued fractional Brownian motions with Hurst parameter $\be$, conditioned 
that 
\begin{align}
B_{-n} = \cj{B_n}, \qquad n \in \Z^d, 
\label{fBM1}
\end{align}

\noi
which in particular implies that $B_0$ is real-valued.
Then,  $\{B_0\}\cup \{\sqrt 2\Re B_n, \sqrt 2 \Im B_n \}_{n\in (\Z^d)_{+}}$ 
forms a family of  mutually independent real-valued fractional Brownian motions
with the same Hurst parameter $\be$. 

%
%\noi
%Then,  $\{\sqrt 2\Re B_n, \sqrt 2 \Im B_n \}_{n\in (\Z^d)_{+, 0}(\Z^d)_{+, 0}}$ 
%forms a family of  mutually independent real-valued fractional Brownian motions
%with the same Hurst parameter $\be$. 

In order to define stochastic integrals with respect to these fractional Brownian motions, 
we need the following {\it real} Hilbert space.
For $\frac 12 < \be< 1$,  let 
$\Hs^\be(\R_+)$ be the completion of 
linear combinations of 
(real-valued) step functions on $\R_+$ under 
the following norm:\footnote{On the class of functions $f$ on $\R$ with 
$\supp (f) \subset \R_+$, \eqref{BM0} defines a norm, not a semi-norm.}
\begin{align}
\label{BM0}
\begin{split}
\|f\|_{\Hs^\be(\R_+)}^2
&   = \be(2\be-1)
\int_0 ^\infty \int_0^\infty f(t) f(t') |t-t'|^{2\be-2} dtdt' \\
& = C_\be \int_0^\infty f(t) \mathfrak{I}_{2\be - 1}(f) (t) dt \\
&  = C_\be  \int_\R  |\tau|^{1-2\be} |\ft{f}(\tau)|^2 d\tau 
= C_\be \|f\|_{\dot H^{\frac 12 - \be}(\R)}^2
\end{split}
\end{align}

\noi
for a function $f$ supported on $\R_+$, 
where
$\mathfrak{I}_{2\be-1} = |\dt|^{1-2\be}$ denotes the Riesz potential of order $2\be - 1>0$.
Similarly, given $k \in \N$, 
we 
define $\Hs^\be(\R_+^k)$ 
(with the understanding 
that $\R_+^k := (\R_+)^k$)
to be
the completion of 
linear combinations 
of products of step functions in  $t_j \in \R_+$, $j = 1, \dots, k$,  under 
\begin{align}
\begin{split}
\|f\|_{\Hs^\be(\R_+^k)}^2
&   = \be^k(2\be-1)^k
\int_{\R_+^{2k}}
f(t_1, \dots, t_k) 
f(t_1', \dots, t_k')
\prod_{j = 1}^k |t_j-t_j'|^{2\be-2} dt_jdt_j' \\
&
= C_\be^k
\bigg\| \prod_{j = 1}^k |\dd_{t_j}|^{\frac 12 - \be}f\bigg\|_{L^2(\R^k_+)}^2
\end{split}
\label{BM0a}
\end{align}

\noi
for a function $f(t_1, \dots, t_k)$ supported on $\R_+^k$. 
When  $\be = \frac12$, we set  
\begin{align}
\Hs^\frac12(\R^k_+) = L^2(\R_+^k).
\label{BM1}
\end{align}

%
%
%Let 
%$\{B_n\}_{n\in \Z^d}$ 
%be the family of independent  complex-valued fractional Brownian motion
%conditioned that 
%
%
%\noi

We say that 
a sequence $f = \{ f_n \}_{n \in \Z^d}$ of complex-valued functions
$f_n$ on $\R_+$
 belongs to $ \l^2(\Z^d;\Hs^\be(\R_+))$, 
if we have
 \begin{align}
 f_{-n} = \cj{f_n}, \quad n \in \Z^d, 
\label{fBM2} 
 \end{align}

\noi
and 
$\Re f_n, \Im f_n \in \Hs^\be(\R_+)$
(note from \eqref{fBM2}  that $\Im f_0 = 0$)
such that 
\begin{align*}
\|f\|_{\l^2(\Z^d;\Hs^\be(\R_+))}
:\! & = \bigg(\sum_{n \in \Z^d} \| f_n\|_{\Hs^\be(\R_+)}^2\bigg)^\frac 12\\
& = \bigg(\sum_{n \in \Z^d} \|\Re  f_n\|_{\Hs^\be(\R_+)}^2
+ \sum_{n \in \Z^d}  \|\Im  f_n\|_{\Hs^\be(\R_+)}^2\bigg)^\frac 12 < \infty.
\end{align*}

%\noi
%Given $f\in \l^2(\Z^d;\Hs^\be(\R_+))$
% satisfying 
% \begin{align}
% f_{-n} = \cj{f_n}, \quad n \in \Z^d, 
%\label{fBM2} 
% \end{align}
 
\noi
We then define the Wiener integral  of $f$
 with respect to $\{B_n\}_{n\in\Z^d}$
 by setting
\begin{align}
\begin{split}
I_1[f]  
& =  \sum_{n\in \Z^d}J_n^{(r)}(f_n)
+ \sum_{n\in \Z^d}J_n^{(i)}(f_n)\\
& = \sum_{n\in \Z^d} \int _0^\infty f_n(t) d \Re B_n(t)
+ i 
 \sum_{n\in \Z^d}
\int _0^\infty f_n(t) d \Im B_n(t).
\end{split}
\label{I1}
\end{align}

\noi
Note that $I_1[f]$ is real-valued
in view of the conditions \eqref{fBM1} and \eqref{fBM2}.
In \eqref{I1},  each summand $J_n^{(r)}(f_n)$ 
or $J_n^{(i)}(f_n)$ is  understood as a Wiener integral;
namely, $\{J_0^{(r)}(f_0)\}\cup \{\sqrt 2J_n^{(r)}(f_n), \sqrt 2 J_n^{(i)}(f_n)\}_{n \in (\Z^d)_{+}}$ is a family of independent mean-zero Gaussian random variables
with variance $ \|f_n\|_{\Hs^\be}^2$.
In particular,  
the map $I_1$ is an isometry from
$\l^2(\Z^d;\Hs^\be(\R_+))$ into $L^2(\Omega, \s(I_1) , \PP)$, where $\s(I_1)$ is the $\s$-algebra generated by the process\footnote{Here, we view $I_1$ as a process
indexed by  $f \in \l^2(\Z^d;\Hs^\be(\R_+))$.}
$I_1= \{I_1[f]: f \in \l^2(\Z^d;\Hs^\be(\R_+))\}$. 
The process $I_1$ is known as an isonormal Gaussian process associated with the Hilbert space 
$\l^2(\Z^d;\Hs^\be(\R_+))$,  satisfying
\begin{align*}
%\label{BM2}
\E\big[I_1[f] I_1[g]\big] = \jb{f,g}_{\l^2_n \Hs_t^\be }
\end{align*}

\noi
for any $f,g\in 
\l^2(\Z^d;\Hs^\be(\R_+))$;
see \cite[Definition 1.1.1]{Nualart06}.

Given $k \in \N \cup\{0\}$, 
we define the $k$th homogeneous Wiener chaos $\HH_k$ 
to be  the closed linear subspace of $L^2(\Omega)$ generated by
\begin{align*}
\big\{ H_k(I_1[f]) : \, f\in \l^2(\Z^d;\Hs^\be(\R_+)),  \, \|f\|_{\l^2_n\Hs^\be_t} = 1   \big\},
\end{align*}

\noi
where $H_k$ denotes the  Hermite polynomial of degree $k$, 
defined via the following generating function:
\begin{equation*}
 e^{tx - \frac{1}{2} t^2} = \sum_{k = 0}^\infty \frac{t^k}{k!} H_k(x).
 \end{equation*}
	
\noi
For readers' convenience, we write out the first few Hermite polynomials:
\begin{align*}
\begin{split}
& H_0(x) = 1, 
\quad 
H_1(x) = x, 
\quad
H_2(x) = x^2 - 1,   
\quad
 H_3(x) = x^3 - 3 x.
\end{split}
\end{align*}

\noi
The spaces $\HH_k$ and $\HH_j$ are orthogonal when $k\neq j$, 
and the real Hilbert space $L^2(\Omega, \s(I_1), \PP)$ admits the following 
Wiener-Ito decomposition: % as an orthogonal sum of the subspaces $\HH_k$:
\begin{align*}
L^2(\Omega, \s(I_1), \PP) = \bigoplus_{k=0}^\infty \HH_k.
\end{align*}

\noi
Note that given any $f\in \l^2(\Z^d;\Hs^\be(\R_+))$, the stochastic integral $I_1[f]$ is an element of the first 
homogeneous Wiener chaos~$\HH_1$.
We now state the Wiener chaos estimate, which follows from the hypercontractivity of the Ornstein-Uhlenbeck semigroup
 due to Nelson \cite{Nelson2};
 see 
 \cite[Theorem~I.22]{Simon}.

\begin{lemma}[Wiener chaos estimate]
\label{LEM:hyp}
Let $k\in\N$.
Then, given any finite $p \ge 1$ and $F \in \H_k$,  we have
\begin{align*}
\| F \|_{L^p(\Omega)} \le p^{\frac k 2 } \| F \|_{L^2(\Omega)}.
\end{align*}

\end{lemma}

Lastly, we introduce multiple stochastic integrals with respect to the fractional Brownian motions 
$\{B_n\}_{n\in\Z^d}$ with Hurst parameter $\frac 12  \le \be < 1$, satisfying \eqref{fBM1}.
Fix  an integer $k \ge 2$.
Given  $f\in \l^2_{n}( (\Z^d)^{k} ; \Hs^\be(\R_+^{k})  )$, where 
$\Hs^\be(\R_+^{k})$ 
is as in  \eqref{BM0a}
and \eqref{BM1},  
 we define its symmetrization by
\begin{align*}
\Sym(f) (z_1,  \ldots, z_k)  
= \frac{1}{k!} \sum_{\s\in \S_k} f(z_{\s(1)},  \ldots, z_{\s(k)}),
\end{align*}

\noi
where 
$z_j = (t_j, n_j)$
and $\S_k$ denotes the symmetric group on $\{1, \dots, k\}$.
We denote by  
$\big(\l^2( (\Z^d)^{k} ; \Hs^\be(\R_+^{k})  )\big)^{ \Sym}$  the subspace of symmetric functions in 
$\l^2( (\Z^d)^{k} ; \Hs^\be(\R_+^{k})  )$.

We now introduce 
 the notion of a multiple Wiener integral $I_k$.

\begin{definition} \rm
Let $k\in\N$.
The $k$th multiple Wiener integral $I_k$ is an isometry 
(up to a constant factor; see \eqref{ISO1})
from 
$\big(\l^2( (\Z^d)^{k} ; \Hs^\be(\R_+^{k})  )\big)^\Sym$
into the Wiener chaos $\HH_k$, uniquely determined by
\begin{align*}
I_k[ h_1^{\otimes k_1} \otimes \cdots \otimes h_n^{\otimes k_n}  ] = \prod_{j=1}^n H_{k_j} (I_1[h_j])
\end{align*}

\noi
for any orthonormal elements $h_1, \ldots, h_n \in 
\l^2(\Z^d;\Hs^\be(\R_+))$ and  $k_1, \dots, k_n \in \N\cup\{0\}$
such that  $k = k_1 + \ldots + k_n$, 
where $H_{k}$ denotes the  Hermite polynomial of degree $k$ and $I_1$ is as in~\eqref{I1}.

\smallskip

\begin{itemize}
\item
For general (possibly non-symmetric)  $f\in
\l^2( (\Z^d)^{k} ; \Hs^\be(\R_+^{k})  )$, 
we set 
\[I_k[f] = I_k [ \Sym(f)].\]

\smallskip
\item
Given $f\in
\l^2( (\Z^d)^{k} ; \Hs^\be(\R_+^{k})  )$
 and $g\in \l^2( (\Z^d)^\l ; \Hs^\be(\R_+^{\l})  )$
for some $k, \l \in \N$, we have
\begin{align}
\E \big[ I_k[f] I_\l[g] \big] = k! \cdot \ind_{k=\l} \cdot \jb{\Sym(f), \Sym(g)}_{\l^2_{n_1, \dots, n_k} \Hs^\be_{t_1, \dots, t_k}}, 
\label{ISO1}
\end{align}

\noi
where
$\l^2_{n_1, \dots, n_k} \Hs^\be_{t_1, \dots, t_k}$
is a short-hand notation for 
$\l^2( (\Z^d)^{k} ; \Hs^\be(\R_+^{k})  )$.

\smallskip
\item
When $\be=\frac{1}{2}$, the multiple Wiener integrals agree with the iterated Wiener-Ito integrals with respect to 
a family  $\{B_n \}_{n \in \Z^d}$ of mutually independent standard Brownian motions, 
satisfying \eqref{fBM1}.
Furthermore, suppose that $f$ is symmetric.
Then, we have 
\begin{align*}
& I_k[f] 
 = k! \sum_{n_1, \cdots, n_k \in \Z^d} 
%\int_{0}^\infty \int_{0}^{t_1} \cdots \int_{0}^{t_{k-1}} 
\intt_{\Dl_k}
f(t_1, n_1, \dots, t_k, n_k)
 dB_{n_k}(t_k) \cdots dB_{n_1}(t_1),
\end{align*}

\noi
where 
$\Dl_k
= \big\{ (t_1, \ldots, t_k) \in \R_+^k: \ t_i > t_j 
\text{ for } i < j\big\}$.
Here,  the iterated integral on the right-hand side is understood as an iterated 
Ito integral; see \cite[p.\,23]{Nualart06}.

\end{itemize}

\end{definition}

\subsection{Random tensor estimate}
\label{SUBSEC:RT}

In this subsection, we provide the basic definition of tensors and 
state the random tensor estimate (Lemma \ref{LEM:RT}).
See~\cite[Sections 2 and 4]{DNY3}, \cite[Section~4]{Bring}, 
\cite[Appendix C]{OWZ}, 
and \cite[Section 2]{OW3}
for further discussions.

\begin{definition} \label{DEF:tensor} \rm
Let $A$ be a finite index set. We denote by $n_A$ the tuple $ \{n_j \}_{j \in A}$. 
 A tensor $h = h_{n_A}$ is a function: $(\Z^d)^{A} \to \mathbb{C} $ with the input variables $n_A$. Note that the tensor $h$ may also depend on $\o \in \Om$. 
 The support of a tensor $h$ is the set of $n_A$ such that $h_{n_A} \neq 0$. 

Given a finite index set  $A$, 
let $(B, C)$ be a partition of $A$. We define the norms 
 $\| \cdot \|_{n_A}$ and 
$\| \cdot \|_{n_{B} \to n_{C}}$ by 
\[ \| h \|_{n_A}  = \|h\|_{\l^2_{n_A}} = \bigg(\sum_{n_A} |h_{n_A}|^2\bigg)^\frac{1}{2}\]
and
\begin{align*}
  \| h \|^2_{n_{B} \to n_{C}} = \sup \bigg\{ 
\sum_{n_{C}} \Big| \sum_{n_{B}} h_{n_A} f_{n_{B}} \Big|^2 :  \| f \|_{\l^2_{n_{B}}} =1  \bigg\},  
%\label{Z0a}
\end{align*}

\noi
where  we used the short-hand notation $\sum_{n_Z} = \sum_{n_Z \in (\Z^d)^Z}$ for a finite index set $Z$.
By duality, we have  $\| h \|_{n_{B} \to n_{C}} = \| h \|_{n_{C} \to n_{B}} 
= \| \cj h \|_{n_{B} \to n_{C}}$ for any tensor $h = h_{n_A}$. 
If $B = \varnothing$ or $C = \varnothing$,  then we have
$  \| h \|_{n_{B} \to n_{C}} = \| h \|_{n_A}$.
\end{definition}

We  now state the random tensor estimate
for multiple stochastic integrals
from \cite{OWZ, CLO2, COZ}.
The random tensor estimate 
was first introduced 
in a breakthrough work \cite{DNY3}
by Deng, Nahmod, and Yue 
in the context of random variables;
see also 
\cite{Bring, OW3}.
See also  \cite{BO96} for a precursor of
the random tensor estimate.
In \cite{OWZ}, 
the third author with Wang and Zine
extended it 
to treat the case of iterated Wiener-Ito integrals
(namely, with the Hurst parameter $\be = \frac 12$), 
which was further extended
in~\cite{CLO2, COZ}
to the case of 
multiple stochastic integrals
with respect to fractional Brownian motions of Hurst parameter $\frac12 < \be <  1$.
In the following, we state a slightly simplified
version of the random tensor estimate from \cite{CLO2},
which is sufficient for our purpose.
See \cite{CLO2}
for a more general statement and its proof
(which follows closely the presentation in \cite{OWZ}, 
corresponding to the $\be = \frac 12$ case).

\begin{lemma}
\label{LEM:RT}

Fix $\frac 12 \le \be < 1$
and let $A$ be a finite index set with $k = |A| \ge 1$. 
Given a  tensor  $h=h_{bc n_A} \in \l^2_{bcn_A}$ 
with  $n_A \in (\Z^d)^{A}$ and $(b,c) \in (\Z^d)^{m}$ for some integer $m\ge2$, 
satisfying
\begin{align*}
h_{-b, -c, -n_A} = \cj{h_{bc n_A}}, 
\quad 
(b,c) \in (\Z^d)^{m}, \ 
n_A \in (\Z^d)^{A}, 
%\label{BM3}
\end{align*}

\noi
 define the random tensor $H = H_{bc}$  by
\begin{align*}
H_{bc} =   \sum_{n_A} I_k \big[ h_{bcn_A} f_{bc} (t_A, n_A) \big]
\end{align*}

\noi
for $f\in \l^\infty_{bc n_A}( (\Z^d)^{k+m} ; \Hs^\be(\R_+^{k})  )$, satisfying
\begin{align*}
f_{-b, -c, -n_A} = \cj{f_{bc n_A}}, 
\quad 
(b,c) \in (\Z^d)^{m}, \ 
n_A \in (\Z^d)^{A}, 
%\label{BM4}
\end{align*}

\noi
where 
$\Hs^\be(\R_+^{k})$ 
is as in  \eqref{BM0a} and \eqref{BM1}
\textup{(}see also \eqref{BM0}\textup{)}
and 
$I_k$ denotes the multiple stochastic integral 
defined in Subsection~\ref{SUBSEC:FBM}.
Then, given any $\ta > 0$
and finite $p \ge1$, we have
\noi
\begin{align*}
\big\| \| H_{bc} \|_{b \to c} \big\|_{L^p(\Om)}
&
\les p^\frac k2
\|f_{bc}(t_A,n_A)\|_{\l^\infty_{bc n_A}\Hs^\be_{t_A}}
\| h_{bcn_A} \|_{\l^2_{bcn_A}}^{\ta}\\
 &
 \quad \times \Big( \max_{(B,C)} \|    h_{bcn_A} \|_{b n_B \to c n_C} \Big)^{1 - \ta} , 
\end{align*}

\noi
where the maximum is taken over  all partitions $(B,C)$ of $A$.

\end{lemma}

We note that the Wiener chaos estimate (Lemma \ref{LEM:hyp})
plays a crucial role in the proof of Lemma~\ref{LEM:RT}.

\begin{remark} \label{REM:RT}\rm

In the proofs of Propositions \ref{PROP:drive2}, \ref{PROP:drive3},
and \ref{PROP:drive4} to study regularity properties
of the drivers $\XX^j$, $j = 1, 2, 3$,   
we apply 
 the random tensor estimate (Lemma \ref{LEM:RT}).
In estimating the $\Hs_{t_A}^\be$-norm, 
we first apply (multi-parameter) Sobolev's inequality
and bound it by the $L^\frac1\be$-norm; see
\eqref{YS1}, \eqref{ZYS1}
and \eqref{ZZYS3}.
In the current parabolic setting, 
this suffices for our purpose since there is no  oscillatory
cancellation to exploit.
In the dispersive setting \cite{CLO2, COZ, CGLLO2, OShao, CLOO}, 
however, such an approach based on Sobolev's inequality 
would be too crude
and  we need to 
estimate the $\Hs_{t_A}^\be$-norm directly 
by exploiting 
subtle oscillatory
cancellations.

\end{remark}

\section{Young case}
\label{SEC:Young}

In this section, we consider the Young case (= the fractional-in-time case with 
the Hurst parameter $\frac 12 < \be < 1$).
As mentioned in Section \ref{SEC:1}, 
our main task is to give a pathwise meaning
to the stochastic convolution $\Psi(u)$
as the convolution Young integral $\I^{\XX^1}(u)$
with the first order driver $\XX^1$ in \eqref{sto1x},
which allows us to rewrite 
the Duhamel formulation~\eqref{NLH2}
as the following 
 convolution YDE:
\begin{align}
u(t) = S(t) u_0 + \I^{\XX^1}(u)(t).
\label{yo12}
\end{align}

\noi
In Subsection \ref{SUBSEC:Y1}, 
we briefly go over the construction of a convolution Young integral
for readers' convenience.
In Subsection \ref{SUBSEC:Y2},  
we then use the random tensor estimate (Lemma~\ref{LEM:RT}) 
to study  regularity properties
of the first order driver $\XX^1$ in \eqref{sto1x}
(including the case $\be = \frac12$).
In Subsection \ref{SUBSEC:Y3}, 
by making a suitable choice of parameters, 
we prove global well-posedness
of the convolution YDE \eqref{yo12}, 
thus establishing  Theorem \ref{THM:1}.

\subsection{Convolution Young integral}
\label{SUBSEC:Y1}

In this subsection,  we recall the computation from 
\cite[Section 4]{GT10} for readers' convenience.
We first define 
the auxiliary  operator $\YY$    by 
\begin{align}
\begin{split}
\YY_{t, r}(f)
&  =  \int_r^t
S_{t- t'} [  f
\phi dW^\be_{t'} ]
%\YY^j_{t, r} 
%& = \YY_{t, r} \circ \YY^{j-1}_{\bul, r}, \quad j \ge 2, 
\end{split}
\label{Y1}
\end{align}

\noi
for $t> r \ge 0$, 
 where  $f$ is a function on $\T$.
Then, 
from \eqref{psi1} and \eqref{Y1}
with \eqref{dl3}, 
we formally have 
\begin{align}
\big(\ft \updl \Psi(u)\big)_{t, r} =  \YY_{t, r} (u_\bul), 
\label{yo2}
\end{align}

\noi
where $\bul$ denotes the variable of integration.
By replacing $\bul$ in \eqref{yo2} by the left endpoint $r$, we formally have 
\begin{align}
\YY_{t, r} (u_\bul)
=   \XX^1_{t, r} (u_r) 
+ Q_{t, r}, 
\label{yo3}
\end{align}

\noi
where $\XX^1$ is the first order driver defined in \eqref{sto1x}.
Here,  $Q$ is formally given as
$Q_{t, r}  = 
 \YY_{t, r} ((\ft \updl u)_{\bul, r})$
 but this expression does not play any role in the following.

Our goal is to find {\it one} error term $Q$
with sufficient regularity, 
which will allow us to define 
 the convolution Young integral 
$\I^{\XX^1}(u)$
in the pathwise manner
as the unique limit of Riemann-Stieltjes type sums;
see \eqref{yo10} below.

We recall the convolution sewing lemma
(\cite[Theorem 3.5]{GT10}), 
which will be the main tool for constructing the convolution Young integral
$\I^{\XX^1}(u)$
(and also the convolution rough integral
$\I^{ \XX^1, \XX^2, \XX^3}(u)$ in Section \ref{SEC:rough}).
See also 
\cite[Theorem 2.4]{GH}
and 
\cite[Exercise 4.16]{FH20}.

\begin{lemma}[convolution sewing lemma]\label{LEM:sew2}
Let $\g > 1$, $s \in \R$, and $T > 0$.

\smallskip
\begin{enumerate}

\item[(i)] 
There exists a unique linear map \textup{(}called the convolution sewing map\textup{)}
$\ft \Ld: C^{\g}_{3,T} H^s(\T) \cap 
\Ker \ft \updl|_{ C_{3,T}H^s_x}
\to C^{\g}_{2,T}H^s(\T)$ such that 
\begin{align*}
\ft \updl \ft \Ld h = h
%\label{CS1}
\end{align*}

\noi
for each  $h\in  C_{3,T}H^s(\T^d)\cap \Ker \ft  \updl|_{ C_{3,T} H^s_x}$, 
satisfying 
\begin{align*}
 \|\ft \Ld h\|_{C^{\g-\ta}_{2, T}H^{s+2\ta}_x}
\les \| h\|_{C^\g_{3, T} H^s_x}
%\label{CS2}
\end{align*}

\noi
for any $0 \le \ta < 1$ with $\ta \le \g$.

\smallskip
\item[(ii)] 
Given any  $g\in  C_{2,T}H^s(\T)$ 
with 
$\ft \updl g\in C^\g _{3,T}H^s(\T)$, 
 there exists  unique
$f\in C([0, T];H^s(\T))$  \textup{(}modulo an additive  constant\textup{)} 
such that 
$\ft \updl f = (\Id - \ft \Ld \ft  \updl)g$. 
In addition, 
we have 
\begin{align*}
(\ft \updl f)_{t,r} = \lim_{|\Pi([r,t])|\to 0} 
\sum_{j=0}^{n-1} S_{t- t_j} g_{t_j,t_{j+1}}
%\label{sew2}
\end{align*}

\noi
 for any $(t,r)\in \Dl_{2,T}$,
 where 
 the limit is over any partition
 $\Pi([r,t])$  
 of  $[r,t]$\textup{:} 
\[\Pi ([r,t]) = \{r = t_{n} < \dots < t_1 <  t_0 = t\}\]
whose mesh size 
$|\Pi([r,t])| = \max_{j} |t_j-t_{j+1}|$ 
tends to $0$.

\end{enumerate}

\end{lemma}

\medskip

Let $\XX^1$ be as in \eqref{sto1x}.
Then, we have
\begin{align}
\wt \updl \XX^1 = 0, 
\label{chen3}
\end{align}

\noi
where $\wt \updl$ is as in \eqref{dl4}.
{\it Suppose} that  we have 
\begin{align}
\XX^1 (\o)
\in  C_{2, T}^{\g} \L(H^s(\T))
\label{yo4}
\end{align}

\noi
for some $s \in \R$,  $\frac 12  < \g < 1$, and $\o \in \Om$.
In the following discussion, we suppress the $\o$-dependence.
In addition, 
 we assume that 
\begin{align}
u \in \ft \CC_{T}^\al H^s(\T)
\label{yo5}
\end{align}

\noi
for some $\frac 12<  \al < 1$, 
where 
$\ft \CC_{T}^\al H^s(\T)$ is as in \eqref{Ho4}.

By applying 
the convolution coboundary operator $\ft \updl$ in \eqref{dl3}
to \eqref{yo3}
and noting that, in view of  \eqref{yo2},  the left-hand side of \eqref{yo3}
formally vanishes under the application of $\ft \updl$, 
any error term $Q$ (if it exists) satisfies
\begin{align*}
\ft \updl Q 
=  - \ft \updl \big(\XX^1 (u) \big)
= \XX^1 (\ft \updl u ) 
 \in C_{3, T}^{\al + \g}H^s(\T), 
\end{align*}

\noi
where the second equality follows from 
\eqref{prod2}
and \eqref{chen3}.
By the assumptions \eqref{yo4} and \eqref{yo5}, 
we have $\al + \g > 1$.
Thus, 
 we can apply the convolution sewing lemma (Lemma \ref{LEM:sew2})
to {\it define} an error term $Q$ by the relation:
\begin{align}
Q =   - \ft \Ld \ft \updl 
\big(\XX^1 (u) \big)
\in C_{2, T}^{\al + \g}H^s(\T), 
\label{yo6}
\end{align}

\noi
where $\ft \Ld$ denotes the convolution sewing map.
Therefore, from \eqref{yo2}, \eqref{yo3}, and  \eqref{yo6}, 
we  make sense of  the stochastic convolution $\Psi(u)$
as the convolution Young integral $\I^{\XX^1}(u)$
of~$u$ (with respect to the Young driver $\XX^1$ in \eqref{yo4}):
\begin{align}
\Psi(u) = \I^{\XX^1}(u)
\label{yo7}
\end{align}

\noi
where $\I^{\XX^1}(u)(0) = 0$
and its convolution increment is given by 
\begin{align}
\ft \updl \I^{\XX^1}(u)
= 
(\Id - \ft \Ld \ft \updl)\big(\XX^1 (u) \big).
\label{yo9}
\end{align}

\noi
In view of \eqref{yo6} with $\al + \g > 1$, 
the convolution rough integral 
$\I^{\XX^1}(u)$ is given by the unique limit of
Riemann-Stieltjes type sums:
\begin{align}
\begin{split}
\I^{\XX^1}(u)(t) 
& =  \lim_{|\Pi([0,t])|\to 0} 
\sum_{j=0}^{n-1} S_{t- t_j} 
\XX^1_{t_j,t_{j+1}} (u_{t_{j+1}}), 
\end{split}
\label{yo10}
\end{align}

\noi
where the limit is understood in the sense of Lemma \ref{LEM:sew2}\,(ii).
Here, we used the fact that 
\begin{align}
 \lim_{|\Pi([0,t])|\to 0} 
\bigg\|\sum_{j=0}^{n-1} S_{t- t_j} 
Q_{t_j,t_{j+1}}\bigg\|_{H^s_x} 
\le
 \lim_{|\Pi([0,t])|\to 0} 
\sum_{j=0}^{n-1} 
\|Q_{t_j,t_{j+1}}\|_{H^s_x} 
= 0, 
\label{yo11}
\end{align}

\noi
provided that $\al + \g > 1$.

In view of \eqref{yo7}, the Duhamel formulation  \eqref{NLH2}
reduces to the  convolution YDE \eqref{yo12}.
The main task that remains is to verify that \eqref{yo4} holds 
for the first order driver $\XX^1$ defined in \eqref{sto1x}, 
almost surely, 
under suitable assumptions on the relevant parameters.

\subsection{First order driver}
\label{SUBSEC:Y2}

In this subsection, 
we study regularity properties
of the first order driver~$\XX^1$ in~\eqref{sto1x}.
From 
\eqref{W1} and \eqref{phi1}, we have 
\begin{align}
\XX^1_{t, r}(f)
&
 = \sum_{n\in \Z} e_n 
 \int_r^t \sum_{n_1, n_2  \in \Z} \ind_{n = n_{12}}
\cdot  e^{-(t-t')\jb{n}^2}e^{-(t'-r)\jb{n_2}^2} \ft f(n_2) \phi_{n_1}    d B_{n_1}(t')
\label{stoconv1}
\end{align}

\noi
for $t > r \ge 0$, 
where
 $n_{12} = n_1 + n_2$ as in \eqref{short1}.
A direct computation shows that $\XX^1$ satisfies~\eqref{chen3}.

\begin{proposition}
\label{PROP:drive2}

Let  $\frac 12 \le \be < 1$
and $s, s_0, \s \in \R$.
Given 
 $\phi \in \HS(L^2(\T); H^\s(\T))$ satisfying~\eqref{phi1}, 
let  $\XX^1$ be 
the first order driver in \eqref{stoconv1}.
Suppose that 
\begin{align}
\begin{split}
s + \s & > - 2a \be,\\
s_0 & < \min \Big( s\wedge 0 + \s + 2a\be, 
\, s + \s\wedge 0  + 2a \be  \Big)
\end{split}
\label{YS0a}
\end{align}

\noi
for some  $0 \le a < 1$.
 Then, given any finite $p \ge 1$, we have 
\begin{align}
\big\| \|  \XX^1_{t,r}  \|_{\LOP(H^s; H^{s_0})} \big\|_{L^p(\Om)} 
 & \les  p^\frac 12 
\| \phi\|_{\HS(L^2;H^\s)} 
  |t-r|^{(1-a)\be}
\label{YG1}
\end{align}

\noi
for any   $t > r\ge 0$.
Consequently, 
given $0 < \g < (1-a)\be$, 
we have 
\begin{align}
\big\| \| \XX^1\|_{C^{\g}_{2,T} \LOP(H^s;H^{s_0})} \big\|_{L^p(\Om)} \les_T 
p^\frac 12 
 \| \phi\|_{\HS(L^2;H^\s)}
\label{YG1a}
\end{align}

\noi
for any finite $p \ge1$ and $T > 0$.
In particular, 
there exists a version of~$\XX^1$
such that 
\[\XX^1 \in C^{\g}_{2,T} \LOP(H^s(\T);H^{s_0}(\T)), \]

\noi
almost surely.

\end{proposition}

\begin{proof} %[Proof of Proposition \ref{PROP:drive2}]
We first note that the bound \eqref{YG1a} follows from \eqref{YG1} and
the Garsia-Rodemich-Rumsey inequality
(see, for example, \cite[Lemma 2.2]{GKOT})
by arguing as in the proof of 
\cite[Lemma 2.3]{GKOT}.
Hence, we focus on proving 
\eqref{YG1} in the following.

%We first note that the bound \eqref{YG1a} follows from \eqref{YG1} and
%the Garsia-Rodemich-Rumsey inequality
%(\cite[Lemma 3.8]{GT10}; see also  \cite[Lemma 2.2]{GKOT}) with \eqref{chen3}.
%Hence, we focus on proving 
%\eqref{YG1} in the following.

Fix $\frac 12 \le \be < 1$
and $t > r \ge 0$.
From \eqref{stoconv1} with \eqref{I1}, we have 
\begin{align}
\Ft_x\big(\jb{\nb}^{s_0}\XX^1_{t,r} \jb{\nb}^{-s}f \big)(n) 
=  \sum_{n_2\in \Z} \ft f (n_2) I_1[ \hf^1_{nn_2}(n_1) \ff^1_{nn_2}(t') ] , 
\label{YG2}
\end{align}

\noi
where $\hf^1_{nn_2}(n_1) $ and $\ff^1_{nn_2}$ are defined by 
\begin{align}
\begin{split}
\hf^1_{n n_2}(n_1) & = 
 \ind_{n=n_{12}}
\cdot
\frac{\jb{n}^{s_0}}{\jb{n_2}^{s}} \phi_{n_1} , \\
\ff^1_{nn_2} (t')  &= 
\ff_{nn_2}^{1, t, r} (t')= 
 \ind_{[r, t]}(t')\cdot 
 e^{-(t-t') \jb{n}^2}e^{-(t' -r) \jb{n_2}^2}.
\end{split}
\label{YG3}
\end{align}

\noi
Given a dyadic triple $\Nb_2= (N, N_1, N_2) \in (2^{\Z_{\ge 0}})^3$, we set 
\begin{align}
\begin{split}
\hf^{1, \Nb_2}_{n n_2}(n_1) & = 
\hf^{1, N, N_1, N_2}_{n n_2}(n_1)
= \ind_{E_{\Nb_2}}
\cdot 
 \hf^1_{n n_2}(n_1), \\
\ff^{1, \Nb_2}_{nn_2} (t')  &= 
\ff^{1, N, N_1, N_2}_{nn_2} (t')  = 
\ind_{E_{\Nb_2}}
\cdot 
\ff^1_{nn_2} (t'),  
\end{split}
\label{YG3a}
\end{align}

\noi
where $E_{\Nb_2}$ is defined by 
\begin{align*}
\begin{split}
E_{\Nb_2} = \big\{(n, n_1, n_2) \in \Z^3:
\ &  n = n_1 + n_2, 
\\
& 
 |n|\sim N, \, |n_j|\sim N_j, \, j = 1, 2\big\}.
\end{split}
%\label{EN1}
\end{align*}

%Given dyadic numbers $N, N_1, N_2\ge 1$, 
%we use $N_{\max}$, $N_{\med}$, and $N_{\min}$
%to denote their decreasing rearrangement:
%\begin{align}
%N_{\max} \ge N_{\med} \ge N_{\min}.
%\label{ord1}
%\end{align}
%

\noi
On $E_{\Nb_2}$, we have
\begin{align}
\max(N, N_1, N_2) \sim \med (N, N_1, N_2), 
\label{EN1a}
\end{align}

\noi
where the right-hand side denotes the second largest
among $N$, $N_1$, and $N_2$.
We define  $\D_2$ 
by setting
\begin{align}
\D_2 = \big\{
\Nb_2 =  (N, N_1, N_2)\in (2^{\Z_{\ge 0}})^3: \, \text{$\Nb_2$ satisfies \eqref{EN1a}}
\big\}.
\label{EN1b}
\end{align}

Given any $\ta > 0$, 
it follows from \eqref{YG2}, \eqref{YG3}, \eqref{YG3a}, 
and the random tensor estimate (Lemma~\ref{LEM:RT})
with \eqref{phi1a}
that 
\begin{align}
\begin{split}
 \big\| &  \| \XX^1_{t,r} \|_{\LOP(H^s;H^{s_0})} \big\|_{L^p(\Om)}
 = \big\| \| \jb{\nb}^{s_0}\XX^1_{t,r}\jb{\nb}^{-s} \|_{\LOP(L^2;L^2)} \big\|_{L^p(\Om)}\\
& = \big\| \| I_1[ \hf^1_{nn_2}(n_1) \ff^1_{nn_2}(t')  ] \|_{\l^2_{n_2} \to \l^2_n} \big\|_{L^p(\Om)}\\
& \le \sum_{\Nb_2 \in \D_2}
\big\| \| 
I_1[ \hf^{1, \Nb_2}_{nn_2}(n_1) \ff^{1, \Nb_2}_{nn_2}(t')] \|_{\l^2_{n_2} \to \l^2_n} \big\|_{L^p(\Om)}\\
& \les 
p^\frac 12 
\sum_{\Nb_2\in \D_2}
\frac{N_{\max}^{\frac 12 \ta} N^{\ta s_0}}
{N_1^{ \ta\s} N_2^{\ta s}}
\|\phi\|_{\HS(L^2; H^\s)}^\ta
 \|  \ff^{1, \Nb_2}_{nn_2} (t')\|_{\l^\infty_{nn_2} \Hs_{t'}^\be}\\
& \hphantom{XXXXX}\times 
 \max\Big( \|  \hf^{1, \Nb_2}_{nn_2}(n_1)\|_{n_1n_2 \to n} , 
 \|   \hf^{\1, \Nb_2}_{nn_2}(n_1) \|_{n_2 \to n n_1} \Big)^{1-\ta}
\end{split}
\label{YG3b}
\end{align}

\noi
for any finite $p \ge 1$, 
where 
 $N_{\max} = \max(N, N_1, N_2)$ and 
 $\Hs^\be_{t'}(\R_+)$ is as in \eqref{BM0}.

From \eqref{YG3a}, \eqref{YG3}, and 
Cauchy-Schwarz's inequality with \eqref{phi1a}, we have
\begin{align}
\begin{split}
\|  \hf^{1, \Nb_2}_{nn_2}(n_1)\|_{n_1n_2 \to n} 
&\sim 
\frac{N^{s_0}}{N_2^s}
\sup_{\|f\|_{\l^2_{n_1 n_2}}
= \|g\|_{\l^2_{n}}  = 1}
\bigg|
\sum_{n, n_1, n_2 \in \Z}
\ind_{E_{\Nb_2}}
\cdot 
\phi_{n_1} f_{n_1n_2} g_{n}\bigg|\\
&\les
\frac{N^{s_0}}{N_1^\s N_2^s}
\|\phi\|_{\HS(L^2; H^\s)}.
\end{split}
\label{YG3x}
\end{align}

\noi
A similar computation yields
\begin{align}
\begin{split}
\|  \hf^{1, \Nb_2}_{nn_2}(n_1)\|_{n_2 \to n n_1} 
&\les
\frac{N^{s_0}}{N_1^\s N_2^s}
\|\phi\|_{\HS(L^2; H^\s)}.
\end{split}
\label{YG3y}
\end{align}

Next,  we  estimate the 
$\Hs^\be_{t'}$-norm of 
$ \ff_{nn_2}^{1, \Nb_2}(t')$
appearing in \eqref{YG3b}.
From  \eqref{BM0} and Sobolev's inequality, we have  
\begin{align}
 \|  \ff^1_{nn_2} (t')\|_{\l^\infty_{nn_2} \Hs_{t'}^\be}
=  
\|  \ff^1_{nn_2} (t')\|_{\l^\infty_{nn_2} \dot H_{t'}^{\frac 12 - \be}}
\les 
\|  \ff^1_{nn_2} (t')\|_{\l^\infty_{nn_2}  L^\frac 1\be_{t'} }.
\label{YS1}
\end{align}

\noi
Before proceeding further, recall 
the beta function  $B(z_1, z_2)$ defined by 
\begin{align}
B(z_1, z_2) = \int_0^1 (1-t)^{z_1 - 1} t^{z_2-1} dt, 
\label{YS2}
\end{align}

\noi
which is finite for $\Re z_1, \Re z_2 > 0$.
Then, 
it follows from \eqref{YG3}, 
the boundedness of $\tau^\ta e^{- \tau} \le C_\ta$ 
for any $\tau \ge 0$
(given $\ta \ge  0$), 
and a change of variables 
with \eqref{YS2}
 that 
\begin{align}
\begin{split}
\|  \ff^1_{nn_2} (t')\|_{ L^\frac 1\be_{t'} }^\frac 1 \be 
& = J^1_{t, r}(n, n_2)\\
: \! & = \int_r^t 
 e^{-\be^{-1}(t-t') \jb{n}^2}e^{-\be^{-1}(t' -r) \jb{n_2}^2}
 dt' \\
& \les_{\be, a_1, a_2} 
\jb{n}^{-2a_1}\jb{n_2}^{-2a_2}
\int_r^t (t - t')^{-a_1} 
(t' - r)^{-a_2}  dt'\\
& = (t-r)^{1-a_1 - a_2}\jb{n}^{-2a_1}\jb{n_2}^{-2a_2}
B(1- a_1 , 1- a_2)\\
& \sim_{a_1, a_2}  (t-r)^{1-a_1 - a_2}\jb{n}^{-2a_1}\jb{n_2}^{-2a_2}, 
\end{split}
\label{YS3}
\end{align}

\noi
\noi
uniformly in $n, n_2 \in \Z$
and  $t > r\ge 0$, 
provided that $0 \le a_1, a_2 < 1$.
Hence, we have
\begin{align}
\begin{split}
\|  \ff^{1, \Nb_2}_{nn_2} (t')\|_{ L^\frac 1\be_{t'} }
& \les (t-r)^{(1-a)\be}N_{\max}^{-2a\be}, 
\end{split}
\label{YS3a}
\end{align}

\noi
\noi
uniformly in $n, n_2 \in \Z$, dyadic $N, N_1, N_2 \ge 1$, 
and $t \ge r\ge 0$,
provided that $0 \le a < 1$.

Therefore, putting
\eqref{YG3b}, \eqref{YG3x}, \eqref{YG3y}, 
\eqref{YS1}, 
and \eqref{YS3a} together with \eqref{EN1b}, we obtain
\begin{align}
\begin{split}
  \big\|    \| &  \XX^1_{t,r} \|_{\LOP(H^s;H^{s_0})} \big\|_{L^p(\Om)}\\
& \les 
p^\frac 12 
\|\phi\|_{\HS(L^2; H^\s)}
 |t- r|^{(1-a) \be} %\frac 12 + \eps}
\sum_{\Nb_2\in \D_2}
N_{\max}^{\frac 12 \ta - 2a\be }
\frac{N^{s_0}}{N_1^{\s} N_2^{s}}\\
& \les 
p^\frac 12 
\|\phi\|_{\HS(L^2; H^\s)}
 |t- r|^{(1- a) \be}
\end{split}
\label{YS4}
\end{align}

\noi
for any $0 \le a < 1$, 
provided that \eqref{YS0a}
holds 
and that $\ta, \eps > 0$ are sufficiently small.
Here, the second inequality in \eqref{YS4} follows
from separately considering the cases:
 (a)~$N \sim N_1 \ges N_2$, 
(b)~$N \sim N_2 \gg N_1$, 
and 
(c)~$N_1 \sim N_2 \gg N$
(also depending on the sign of the exponent of $N_{\min}$).
\end{proof}

\subsection{Proof of Theorem \ref{THM:1}}
\label{SUBSEC:Y3}

We conclude this section by presenting a proof of Theorem~\ref{THM:1}.
More precisely, we prove that 
%under a suitable choice of parameters, 
the convolution YDE~\eqref{yo12}
is pathwise globally well-posed.
We note that, 
once we make a suitable choice of parameters, 
the rest follows from a standard pathwise well-posedness
argument for a YDE.
See, for example,  \cite [Subsection 3.3]{CGLLO1}.

%
%
%\noi
%By taking $s \to 0$, 
%we obtain 
%\begin{align}
%\s  >  -   2 \be + 1.
%\label{YS0a2}
%\end{align}
%
%
% 
%

\begin{proof}[Proof of Theorem \ref{THM:1}]

Given $\frac 12 < \be < 1$, 
let $\s \in \R$ satisfy \eqref{TH1}.
Let  $s, \eps > 0$ sufficiently small such that 
\begin{align}
 \s  - s > - 2 \be + 1 + 6\eps.
\label{YR0}
\end{align}

\noi
By taking $\eps > 0$ possibly smaller
 such that $\frac 12 + 3\eps < \be$, 
set 
$a = 1 - \frac1 {2\be} - \frac {3\eps} \be \in (0, 1)$
such that 
\begin{align}
(1-a) \be = \frac 12 + 3\eps.
\label{YY1}
\end{align}

\noi
Then, 
the conditions in  \eqref{YS0a} with $s_0 = s> 0$ and $\s <  0$
reduces to 
\eqref{YR0}, 
which is satisfied
by our choice of parameters.
Hence, 
it follows from Proposition \ref{PROP:drive2} with $s_0 = s$
that there exists $\Si \subset \Om$ with $\PP(\Si) = 1$ such that, 
for each $\o \in \Si$, we have 
\begin{align*}
\XX^1 
= \XX^1 (\o)
\in \bigcap_{j = 1}^\infty C^{\frac 12 +2\eps}_{2,j} \L(H^s(\T)).
%\label{YR2}
\end{align*}

\noi
In the following, we fix $\o \in \Si$
and often suppress dependence on $\o$.

Fix $0 < T \le 1$ (to be chosen later).
Define a map $\G$ on $\ft \CC_{T}^{\frac 12 + \eps} H^s(\T)$ by setting
\begin{align}
\begin{split}
\G(u)(t) 
& = S(t) u_0 + \I^{\XX^1}(u)(t)\\
& = S(t) u_0 + 
\big[(\Id - \ft \Ld \ft \updl)(\XX^1(u))\big]_{t, 0}.
\end{split}
\label{YR3}
\end{align}

\noi
See \eqref{yo9}.
From \eqref{Ho2} with \eqref{prod1}, 
we have 
\begin{align}
\begin{split}
\| \XX^1(u)  \|_{C_{2, T}^{\frac 12 +\eps} H^s_x}
& \le T^\eps \| \XX^1(u)  \|_{C_{2, T}^{\frac 12 +2\eps} H^s_x}
 \le T^\eps \|\XX^1\|_{C^{\frac 12 +2\eps}_{2, 1} \L(H^s)}
\| u \|_{L^\infty_T H^s_x}\\
& \les T^\eps  \|\XX^1\|_{C^{\frac 12 +2\eps}_{2, 1} \L(H^s)}
\| u \|_{\ft \CC_{T}^{\frac 12 +\eps}  H^s_x}.
\end{split}
\label{JL2}
\end{align}

\noi
By applying 
 the convolution sewing lemma (Lemma \ref{LEM:sew2})
with \eqref{Ho2}, we have 
\begin{align}
\begin{split}
 \|\ft \Ld \XX^1(\ft \updl u)\|_{C_{2, T}^{\frac 12 + \eps} H^s_x}
& \le T^{\frac 12 + 2\eps }
\|\ft \Ld \XX^1(\ft \updl u)\|_{C_{2, T}^{1+3\eps} H^{s}_x}\\
& \les
T^{\frac 12 + 2\eps }
\|\XX^1(\ft \updl u)\|_{C_{3, T}^{1+ 3\eps} H^{s}_x}\\
& \les
T^{\frac 12 + 2\eps }
\|\XX^1\|_{C_{2, 1}^{\frac 12 + 2\eps} \L(H^s)}
\| u \|_{\ft \CC_{T}^{\frac 12 + \eps}  H^s_x}.
\end{split}
\label{JL4}
\end{align}

\noi
From \eqref{YR3}, \eqref{JL2}, and \eqref{JL4}
with \eqref{Ho4}, we have
\begin{align}
\begin{split}
\| \G(u)  \|_{\ft \CC_{T}^{\frac 12 +\eps} H^s_x}
& \les \|u_0 \|_{H^s}+  T^\eps  \|\XX^1\|_{C^{\frac 12 +2\eps}_{2, 1} \L(H^s)}
\| u \|_{\ft \CC_{T}^{\frac 12 +\eps}  H^s_x}.
\end{split}
\label{JL4a}
\end{align}

\noi
A similar computation yields the following difference estimate:
\begin{align}
\begin{split}
\| \G(u) - \G(v)  \|_{\ft \CC_{T}^{\frac 12 +\eps} H^s_x}
& \les   T^\eps  \|\XX^1\|_{C^{\frac 12 +2\eps}_{2, 1} \L(H^s)}
\| u -v\|_{\ft \CC_{T}^{\frac 12 +\eps}  H^s_x}.
\end{split}
\label{JL4b}
\end{align}

\noi
Hence, 
by 
choosing 
\begin{align}
T = T\big(  \|\XX^1(\o) \|_{C^{\frac 12 +2\eps}_{2, 1} \L(H^s)}\big) >0
\label{JL5}
\end{align}

\noi
sufficiently small, 
it follows from \eqref{JL4a} and \eqref{JL4b} that 
$\G$ is a contraction on 
$ \ft \CC_{T}^{\frac 12 + \eps} H^s(\T)$. 
Therefore, 
by the Banach fixed point theorem, 
there exists a unique fixed point $u \in \ft \CC_{T}^{\frac 12 + \eps} H^s(\T)$
such that 
\begin{align*}
u (t)=  \G(u)(t) = S(t) u_0 + \I^{\XX^1}(u)(t).
%\label{JL6}
\end{align*}

\noi
Namely, 
$u$ is a unique solution to the convolution YDE \eqref{yo12}.

Thanks to the linearity 
of 
  $\I^{\XX^1}(u)$ 
in $u$, 
the local existence time $T$ in \eqref{JL5}
is  independent of  the initial data $u_0$.
Given a target time $T_* \gg1$, we now choose 
a (smaller) local existence time 
\begin{align*}
T = T\big(  \|\XX^1(\o) \|_{C^{\frac 12 +2\eps}_{2, T_*} \L(H^s)}\big) >0
\qquad \text{such that}\qquad 
T^\eps  \|\XX^1\|_{C^{\frac 12 +2\eps}_{2, T_*} \L(H^s)} \ll 1.
\end{align*}

\noi
%as in \eqref{JL5}.
Then, by iteratively applying 
 the contraction argument 
on each time interval $[jT, (j+1)T]\cap [0, T_*]$, $j = 0,1,  \dots, \big[\frac {T_*}T\big] $, 
where $[x]$ denotes the integer part of $x \in \R$, 
we can extend the solution $u$ to \eqref{yo12} 
onto the entire interval $[0, T_*]$.
This proves 
pathwise global well-posedness.
\end{proof}

\begin{remark}\label{REM:GT1}\rm
In Section \ref{SEC:rough}, 
we treat the rough case ($\be = \frac 12$) by 
considering the third order expansion.
We point out that, 
even in the fractional-in-time case ($\frac 12 < \be < 1$), 
we can improve the regularity restrictions \eqref{TH1}
by working with the third order expansion, 
where we only impose the temporal regularity of $\frac 14 + 2\eps$
on the relevant drivers.
Then, 
instead of \eqref{YY1}, we have 
\begin{align*}
(1-a) \be = \frac 14 + 3\eps, 
\end{align*}

\noi
which yields 
$ \s  - s > - 2 \be + \frac 12 + 6\eps$, 
leading to an improvement of $\frac12$ for the range of $\s$
(regarding boundedness of the first order driver $\XX^1$).
See also Remark \ref{REM:GT2}\,(ii).
We, however,  do not pursue this issue, since our purpose here is to compare our approach
with that in \cite[Section~4]{GT10}.

\end{remark}

\section{Rough case}
\label{SEC:rough}

In this section, we consider the rough case (= the white-in-time case with $\be = \frac 12$).
In this case, the first order driver $\XX^1$ in \eqref{sto1x}
is not sufficient 
and thus we need to augment it by introducing higher order drivers $\XX^j$, $j \ge 2$, 
via the relation:
%We introduce the higher order drivers $\XX^j$, $j \ge 2$,  by the relation:
\begin{align}
\XX^j_{t, r} = \YY_{t, r} \circ \XX^{j-1}_{\bul, r}, 
\label{high1a}
\end{align}

\noi
where $\YY$ is as in \eqref{Y1} and $\bul$ denotes the variable of integration.
Namely, we have 
\begin{align}
\XX^{j}_{t, r}(f) = 
\int_r^t S_{t-t'} \big[ \XX^{j - 1}_{t', r} (f)\phi dW^\be_{t'}\big].
\label{high1}
\end{align}

\noi
Then, Chen's relation in the current convolution setting reads as
\begin{align}
\wt \updl \XX^n = \sum_{j = 1}^{n-1}\XX^j \XX^{n - j}, 
\label{chen1}
\end{align}

\noi
where $\wt \updl $ is as in \eqref{dl4}; see \cite[p.\,40]{GT10}.

By following \cite[Section 5]{GT10}, 
we  consider a third order rough path
$\cj \XX = (\XX^1, \XX^2, \XX^3)$
and 
make sense of  the stochastic convolution $\Psi(u)$
as a convolution rough integral
$\I^{\cj \XX}(u)$.
This  allows us to rewrite 
the Duhamel formulation~\eqref{NLH2}
as the following 
 convolution RDE:
\begin{align}
u(t) = S(t) u_0 + \I^{\cj \XX}(u)(t).
\label{ro1a}
\end{align}

\noi
In Subsection \ref{SUBSEC:R1}, 
we briefly go over the construction of a convolution rough integral
for readers' convenience.
In Subsection \ref{SUBSEC:R2},  
we then use the random tensor estimate (Lemma~\ref{LEM:RT}) 
to study  regularity properties
of the second and third  order drivers $\XX^2$ and $\XX^3$
(for the range $\frac 12 \le \be < 1$).
In Subsection~\ref{SUBSEC:R3}, 
we prove global well-posedness
of the convolution RDE \eqref{ro1a}, 
thus establishing  Theorem~\ref{THM:2}.

\subsection{Convolution rough integral}
\label{SUBSEC:R1}

Following the presentation in \cite[Section 5]{GT10}, 
we consider a second order controlled path
in the current convolution setting.

\begin{definition}\label{DEF:high} \rm
Let $V$ be a separable Banach space and 
$\{S_t\}_{t \in \R_+} \subset \L(V)$ be a semigroup
of bounded operators with $S_0 = \Id$.
Fix  $T > 0$.

\smallskip

\noi
(i) 
Given a pair $(\XX^1, \XX^2)
\in C_{2, T}^\g \L(V)\times  C_{2, T}^{2\g} \L(V)$ of drivers
with $0 < \g < 1$, 
we say that $u \in \ft \CC_{T}^\al V$ for some $0 < \al < 1$ is controlled by 
$(\XX^1, \XX^2)$
if there exist
$u^1, u^2 \in \ft \CC^\al_{T}V$
and remainder terms $R^j \in C_{2, T}^{\al + (2-j)\g} V$, 
%for some $\g_0 > 0$, 
$j = 0, 1$, 
such that 
\begin{align}
\begin{split}
(\ft \updl u)_{t, r} & = \XX^1_{t, r} (u^1_r) + \XX^2_{t, r}(u^2_r) + R^0_{t, r}, \\
(\ft \updl u^1)_{t, r} & = \XX^1_{t, r} (u^2_r)  + R^1_{t, r}
\end{split}
\label{control1}
\end{align}

\noi
for $0 \le r < t \le T$, 
where $\ft \updl$ is as in \eqref{dl3}
with the given semigroup $\{S_t\}_{t \in \R_+}$.
We refer to $u^j$, $j = 1, 2$, 
as the $j$th Gubinelli derivative of $u$
(with respect to the driver $ (\XX^1, \XX^2)$).
We denote the space of such controlled paths 
$(u,u^1, u^2)$ by 
$\D_{\XX^1, \XX^2, T}^{\al, \g}(V)
= \D_{\XX^1, \XX^2}^{\al, \g}([0, T]; V)$
and endow it with the following norm:
\begin{align}
\begin{split}
\| (u, u^1, u^2) \|_{\D^{\al, \g}_{\XX^1, \XX^2,  T}(V)} 
& = \|u(0) \|_{V} + \|u^1\|_{\ft \CC_T^\al V}+ \|u^2\|_{\ft \CC_T^\al V}\\
& \quad 
 + \| R^0 \|_{C^{\al + 2\g}_{2, T}V}
+ \| R^1 \|_{C^{\al + \g}_{2, T}V},
\end{split}
\label{control2}
\end{align}

\noi
where $R^0 $ and $R^1$ are defined by  the relation \eqref{control1}.

\smallskip

\noi
(ii)  Let $0 < \g \le \frac 13$.
We say that  a triplet  $(\XX^1, \XX^2, \XX^3)
\in \prod_{j = 1}^3 C_{2, T}^{j\g} \L(V)$
is a $\g$-H\"older convolution rough path
if it satisfies the Chen's relation
in the current convolution setting:
\begin{align}
\wt \updl \XX^1 = 0, \qquad 
\wt \updl \XX^2 = \XX^1 \XX^1, \qquad 
\wt \updl \XX^3 = \XX^1 \XX^2
+ \XX^2 \XX^1, 
\label{chen2}
\end{align}

\noi
where $\wt \updl $ is as in \eqref{dl4}.
See \eqref{chen1}.

\end{definition}

In our setting, we take $S_t = e^{t(\dx^2- 1)}$
and $V = H^s(\T)$.
As in the Young case discussed in Section \ref{SEC:Young}, 
our goal is to 
provide  a meaning to the stochastic convolution $\Psi(u)$ in 
\eqref{psi1} whose convolution increment is formally 
given by
\begin{align}
\big(\ft \updl \Psi(u)\big)_{t, r} =  \YY_{t, r} (u_\bul), 
\label{ro2}
\end{align}

\noi
where $\YY$ is as in \eqref{Y1}
and $\bul$ denotes the variable of integration.

For $j = 2, 3$, let $\XX^j$ be as in \eqref{high1a}, 
which satisfies
Chen's relation \eqref{chen2}.
{\it Suppose} that 
\begin{align}
\cj \XX (\o)= (\XX^1(\o), \XX^2(\o), \XX^3(\o))
\in \prod_{j = 1}^3 C_{2, T}^{j\g} \L(H^s(\T))
\label{ro2a}
\end{align}

\noi
for some $s \in \R$, $0 < \g \le \frac 13$, 
and $\o \in \Om$.
In the following discussion, we suppress the $\o$-dependence.
In view of \eqref{ro1a}, 
we assume that $u
\in \ft \CC_{T}^\al H^s(\T)$ for some $0 < \al < 1$  is controlled
by $(\XX^1, \XX^2)$ with its Gubinelli derivatives\footnote{In fact, \eqref{ro3}
follows as a result of a contraction argument; see \eqref{ZZ13}.} 
\begin{align}
u^1 = u^2 = u\in \ft \CC_{T}^\al H^s(\T).
\label{ro3}
\end{align}

\noi
Namely, we have
\begin{align*}
\begin{split}
(\ft \updl u)_{t, r} & = \XX^1_{t, r} (u_r) + \XX^2_{t, r}(u_r) + R^0_{t, r}, \\
(\ft \updl u)_{t, r} & = \XX^1_{t, r} (u_r)  + R^1_{t, r}, 
\end{split}
%\label{control3}
\end{align*}

\noi
where 
\begin{align}
R^j \in C_{2, T}^{\al + (2-j)\g} H^s(\T), \quad j = 0, 1.
\label{ro3a}
\end{align}

In following, we recall the discussion  from 
\cite[Section 5]{GT10} for readers' convenience.
For computational clarity, however, we  use $u^1$ and $u^2$ to
denote the Gubinelli derivatives of $u$.
By replacing $\bul$ in \eqref{ro2} by the left endpoint $r$, we have 
\begin{align}
\YY_{t, r} (u_\bul)
=   \XX^1_{t, r} (u_r) 
+ \YY_{t, r} (\ft \updl u_{\bul, r}).
\label{ro4}
\end{align}

\noi
See \eqref{yo3}.
Then, by applying \eqref{control1} and \eqref{high1a} to \eqref{ro4}, 
we have 
\begin{align}
\YY_{t, r} (u_\bul)
=   \XX^1_{t, r} (u_r) 
+  \XX^2_{t, r} (u_r^1) 
+  \XX^3_{t, r} (u_r^2) 
+ Q_{t, r} , 
\label{ro5}
\end{align}

\noi
where $Q = \YY(R^0)$.
As in the Young case, our goal is to find {\it one} error term $Q$
with  sufficient regularity, 
which will allow us to define 
the convolution rough integral 
$\I^{\cj \XX}(u)  = \I^{\XX^1, \XX^2, \XX^3}(u)$
in the pathwise manner
as the unique limit of Riemann-Stieltjes type sums;
see \eqref{ro10} below.

By applying 
the convolution coboundary operator $\ft \updl$ in \eqref{dl3}
to \eqref{ro5}, 
any error term $Q$ (if it exists) satisfies
\begin{align}
\ft \updl Q 
=  - \ft \updl \big[\XX^1 (u) 
+    \XX^2 (u^1) 
+  \XX^3 (u^2) \big], 
\label{ro5a}
\end{align}

\noi
where we used the short-hand notation \eqref{prod1}.
By applying the product rule \eqref{prod2}
(where $u^0:= u$)
with \eqref{chen2} and \eqref{control1}, 
we obtain
\begin{align}
\begin{split}
\ft \updl Q 
& = \sum_{j = 1}^3\Big( - (\wt \updl \XX^j) (u^{j-1}) + \XX^j (\ft \updl u^{j-1})\Big)\\
& = \XX^1\big(\ft \updl u - \XX^1 (u^1) - \XX^2 (u^2)\big)
+ \XX^2\big(\ft \updl u^1 - \XX^1 (u^2)\big)
+ \XX^3(\ft \updl u^2)
\\
& = \XX^1(R^0) + 
\XX^2(R^1)
+ \XX^3(\ft \updl u^2).
\end{split}
\label{ro6}
\end{align}

\noi
See \cite[(90)]{GT10}.
From 
\eqref{ro6}, \eqref{ro2a}, \eqref{ro3}, \eqref{ro3a}
with 
 \eqref{Ho2}, we obtain that 
\begin{align*}
\ft \updl Q  \in C_{3, T}^{\al + 3\g}H^s(\T).
\end{align*}

\noi
Hence, if $\al + 3\g > 1$, 
then we can apply the convolution sewing lemma (Lemma \ref{LEM:sew2})
to {\it define} an error term $Q$ by the relation:
\begin{align}
Q =   - \ft \Ld \ft \updl 
\big[\XX^1 (u) 
+   \XX^2 (u^1) 
+   \XX^3 (u^2) \big]
\in C_{2, T}^{\al + 3\g}H^s(\T), 
\label{ro7}
\end{align}

\noi
where $\ft \Ld$ denotes the convolution sewing map.
Therefore, from \eqref{ro2}, \eqref{ro5}, and  \eqref{ro7}, 
we  can make sense of  the stochastic convolution $\Psi(u)$
as the convolution rough integral $\I^{\cj \XX}(u)$
of~$u$ (with respect to the driver $\cj \XX= (\XX^1, \XX^2, \XX^3)$):
\begin{align}
\Psi(u) = \I^{\cj \XX}(u)
\label{ro8}
\end{align}

\noi
where $\I^{\cj \XX}(u)(0) = 0$
and its convolution increment is given by 
\begin{align}
\ft \updl \I^{\cj \XX}(u)
= 
(\Id - \ft \Ld \ft \updl) \big[\XX^1 (u) 
+   \XX^2 (u^1) 
+   \XX^3 (u^2) \big]. 
\label{ro9}
\end{align}

\noi
In view of \eqref{ro7} with $\al + 3\g > 1$, 
the convolution rough integral 
$\I^{\cj \XX}(u)$ is given by the unique limit of
Riemann-Stieltjes type sums:
\begin{align}
\begin{split}
\I^{\cj \XX}(u)(t) 
& =  \lim_{|\Pi([0,t])|\to 0} 
\sum_{j=0}^{n-1} S_{t- t_j} 
\big[\XX^1_{t_j,t_{j+1}} (u_{t_{j+1}}) \\
& \hphantom{XXXXXXXXXX}
+    \XX^2_{t_j,t_{j+1}} (u^1_{t_{j+1}}) 
+   \XX^3_{t_j,t_{j+1}} (u^2_{t_{j+1}}) \big], 
\end{split}
\label{ro10}
\end{align}

\noi
where the limit is understood in the sense of Lemma \ref{LEM:sew2}\,(ii).
Here, we used \eqref{yo11}
which holds 
for  $\al + 3\g > 1$ in the current setting.

In view of \eqref{ro8}, the 
Duhamel formulation \eqref{NLH2} reduces to 
the  convolution RDE \eqref{ro1a}.
As in the Young case, 
the main task  is then to verify that \eqref{ro2a} holds,  
almost surely, 
under suitable assumptions on the relevant parameters.

\subsection{Higher order drivers}
\label{SUBSEC:R2}

In this subsection, we study regularity properties
of the second and third order drivers $\XX^j$, $j = 2, 3$, 
for $\frac 12 \le \be < 1$.
From \eqref{high1a} with \eqref{sto1x} and \eqref{Y1}, we 
formally have 
\begin{align}
\begin{split}
\XX^2_{t, r}(f)
&  =  \int_r^t
S_{t- t_1} \bigg[\int_r^{t_1} S_{t_1 - t_2} 
 \big[ (S_{t_2 - r} f)
\phi dW^\be_{t_2} \big]
\phi dW^\be_{t_1}\bigg]
\end{split}
\label{high2}
\end{align}

\noi
and 
\begin{align}
\begin{split}
\XX^3_{t, r}(f)
&  =  \int_r^t
S_{t- t_1} \bigg[\int_r^{t_1} S_{t_1 - t_2} 
\Big[\int_r^{t_2} S_{t_2 - t_3} \\
& \hphantom{XXXXXX}
 %\times 
 \big[ (S_{t_3 - r} f)
\phi dW^\be_{t_3} \big]
\phi dW^\be_{t_2} \Big]
\phi dW^\be_{t_1}\bigg]
\end{split}
\label{high3}
\end{align}

\noi
for $t > r \ge 0$, 
 where  $f$ is a function on $\T$.
Here,  the iterated stochastic integrals in \eqref{high2} and \eqref{high3}
are to be interpreted as  multiple Wiener  integrals
introduced in Subsection~\ref{SUBSEC:FBM}.
See \eqref{RGG1} and \eqref{ZRGG1}.

In view of \eqref{high2} and \eqref{high3}  with  \eqref{W1} and~\eqref{phi1}, 
we can write $\XX^2$ and $\XX^3$ as 
 \begin{align}
\begin{split}
 \XX^2_{t, r}(f)
&
 = \sum_{n\in \Z} e_n 
 \int_r^t
  \int_r^{t_1}
  \sum_{n_1, n_2, n_3  \in \Z} 
\ind_{n = n_{123}} \\
& \hphantom{XXXXXXX}
\times  e^{-(t-t_1)\jb{n}^2}
e^{-(t_1-t_2)\jb{n_{23}}^2}
e^{- (t_2-r)\jb{n_3}^2}
 \ft f(n_3) \\
& \hphantom{XXXXXXX} 
\times 
\prod_{j = 1}^2 \phi_{n_j}     dB_{n_j}(t_j) 
\end{split}
\label{high2a}
\end{align}

\noi
and 
 \begin{align}
\begin{split}
 \XX^3_{t, r}(f)
&
 = \sum_{n\in \Z} e_n 
 \int_r^t
\int_r^{t_1}
\int_r^{t_2}
  \sum_{n_1, \dots,  n_4  \in \Z} 
\ind_{n = n_{1234}} \\
& \hphantom{XXXX}
\times  e^{-(t-t_1)\jb{n}^2}
e^{-(t_1-t_2)\jb{n_{234}}^2}
e^{-(t_2-t_3)\jb{n_{34}}^2}
e^{-(t_3-r)\jb{n_4}^2}
 \ft f(n_4) \\
& \hphantom{XXXX} 
\times 
\prod_{j = 1}^3 \phi_{n_j}     dB_{n_j}(t_j) 
\end{split}
\label{high3a}
\end{align}

\noi
for $t > r \ge 0$,  
where
we used the short-hand notation \eqref{short1}.
Once again, these 
 iterated stochastic integrals in \eqref{high2a} and \eqref{high3a}
are to be interpreted as  multiple Wiener  integrals.

We  note that, 
the drivers $\XX^1$, $\XX^2$, and $\XX^3$
in \eqref{sto1x}, \eqref{high2}, and \eqref{high3}, respectively, 
satisfy Chen's relation 
\eqref{chen2}
since $\XX^j$'s generated by the recursive relation \eqref{high1} satisfy more general Chen's relation
\eqref{chen1}.

\medskip

\noi
$\bul$ {\bf Second order driver $\XX^2$ in \eqref{high2a}.}\rule[-3mm]{0pt}{0pt}\\
\indent
We first introduce some notations.
Given dyadic $N, N_1, N_2, N_3, N_{23} \ge1$, 
we use $\Nb_3$ to denote the following dyadic quintuple:
\begin{align}
\Nb_3 = (N, N_1, N_2, N_3, N_{23}) \in (2^{\Z_{\ge 0}})^5.
\label{XY1}
\end{align}

\noi
Given $\Nb_3 \in (2^{\Z_{\ge 0}})^5$, 
define  $ E_{\Nb_3}$ by 
\begin{align}
\begin{split}
E_{\Nb_3} = \big\{(n, n_1, n_2, n_3) \in \Z^4:
\ & n = n_{123}, \,  |n|\sim N, \\
&  |n_j|\sim N_j, \, j = 1, 2, 3, \ |n_{23}|\sim N_{23}\big\}.
\end{split}
\label{XY2}
\end{align}

\noi
On $E_{\Nb_3}$, we have
\begin{align}
\max( N, N_{23}, N_3) \sim N_{\max} := \max (N, N_1, N_2, N_3).
\label{XY3}
\end{align}

\noi
We define  $\D_3$ 
by setting
\begin{align}
\D_3 = \big\{
\Nb_3 \in (2^{\Z_{\ge 0}})^5: \, \text{$\Nb_3$ is  of the form \eqref{XY1}, satisfying \eqref{XY3}}
\big\}.
\label{XY3a}
\end{align}

\noi
We also denote by  $\A_3$ the class of functions
$\ab = (a_1, a_2, a_3)
:(2^{\Z_{\ge 0}})^5 \to [0, 1)^3$
of the form: 
\begin{align*}
 \ab(\Nb_3) = \big(a_1(\Nb_3), a_2(\Nb_3), a_3(\Nb_3)\big) \in [0, 1)^3
% \label{XY4} 
\end{align*}

\noi
for $\Nb_3 \in (2^{\Z_{\ge 0}})^5$, 
satisfying
\begin{align}
\sum_{j = 1}^3 a_j \le 2.
\label{XY5}
\end{align}

\begin{proposition}
\label{PROP:drive3}
Let  $\frac 12 \le \be < 1$
and $s, s_0, \s \in \R$.
Given $\phi \in \HS(L^2(\T); H^\s(\T))$
satisfying~\eqref{phi1}, 
let  $\XX^2$ be the second order driver in~\eqref{high2a}.
Given
small $\ta > 0$ and 
$\ab \in \A_3$, 
set 
\begin{align}
M^{(2)}_{\be, \s, s, s_0, \ta} (\ab(\Nb_3))
= 
\frac{N_{\max}^{\frac 12 \ta} N^{s_0-2a_1(\Nb_3)\be}}
{N_1^{\s} N_2^\s N_{23}^{2a_2(\Nb_3)\be }N_3^{s+2a_3(\Nb_3)\be}}
\label{XY6}
\end{align}

\noi
for $\Nb_3$ of the form \eqref{XY1}, 
where $N_{\max} = \max(N, N_1, N_2, N_3)$.
 Then, given any finite $p \ge 1$
 and $\ab \in \A_3$,  we have\footnote{Note that the right-hand side
 of \eqref{XY7} is not a priori finite.} 
\begin{align}
\begin{split}
 & \big\| \| \XX^2_{t,r}  \|_{\LOP(H^{s}; H^{s_0})} 
  \big\|_{L^p(\Om)} 
   \les_\ab    p  \| \phi \|_{\HS(L^2;H^\s)}^2  \\
& \hphantom{XXXXXX}\times 
\sum_{\Nb_3 \in \D_3}
%\sum_{\substack{N, N_1, N_2, N_3, N_{23} \ge1 \\\text{dyadic}\\\eqref{XY2}}}  
M^{(2)}_{\be, \s, s, s_0, \ta} (\ab(\Nb_3))
 (t-r)^{(2
-\sum_{j = 1}^3 a_j (\Nb_3))\be}
\end{split}
\label{XY7}
\end{align}

\noi
for any   $t>r\ge 0$. % uniformly in $\ab \in \A_3$.

Furthermore, suppose that the following holds
for a given choice of   $\ab \in \A_3$\textup{:}
\smallskip

\begin{itemize}
\item[(i)]
$M^{(2)}_{\be, \s, s, s_0, \ta} (\ab(\Nb_3))$ in \eqref{XY6} is summable\textup{:}
\begin{align}
\sum_{\Nb_3 \in \D_3}
M^{(2)}_{\be, \s, s, s_0, \ta} (\ab(\Nb_3)) < \infty, 
\label{XY8}
\end{align}

\smallskip

\item[(ii)] 
there exists $0 < \g < 1$ such that 
\begin{align}
\inf_{\Nb_3\in \D_3}\bigg(2
-\sum_{j = 1}^3 a_j (\Nb_3)\bigg)\be>  \g, 
\label{XY9}
\end{align}

\smallskip

\item [(iii)]
for the choice of $\g$ in (ii), the bound \eqref{YG1a} 
for the first order driver $\XX^1$ holds.

\end{itemize}

\smallskip

\noi
Then, 
 we have 
\begin{align}
\big\| \| \XX^2 \|_{C^{\g}_{2,T} \LOP(H^s;H^{s_0})} \big\|_{L^p(\Om)} \les_T p  \| \phi \|_{{\HS} (L^2; H^\s)}^2
\label{XY10}
\end{align}

\noi
for 
any finite $p \ge 1$ and $T > 0$.
In particular, 
there exists a version of $\XX^2$
such that 
\[\XX^2 \in C^{\g}_{2,T} \LOP(H^s(\T);H^{s_0}(\T)),\]

\noi
almost surely.

\end{proposition}

In the next subsection, 
we make a suitable choice of $\ab \in \A_3$, 
guaranteeing \eqref{XY8} and~\eqref{XY9}.

\begin{proof} [Proof of Proposition \ref{PROP:drive3}]

It suffices to prove the bound \eqref{XY7}, 
since the bound \eqref{XY10} follows
from~\eqref{XY7}, \eqref{XY8}, \eqref{XY9},  and 
 \cite[Lemma 3.8]{GT10} with \eqref{chen2} and \eqref{YG1a}
 for $\XX^1$.
Here, we need~\eqref{YG1a}
 for $\XX^1$,  since 
 \cite[Lemma 3.8]{GT10} requires a bound on $\wt \updl \XX^2
 = \XX^1 \XX^1$ (see \eqref{chen2}).

Fix $\frac 12 \le \be < 1$
and  $t > r \ge 0$.
From \eqref{high2a}, we have 
\begin{align}
\Ft_x\big( \jb{\nabla}^{s_0} \XX^2_{t,r} \jb{\nabla}^{-s} f\big) (n) 
= \sum_{n_3 \in \Z} \ft f(n_3) I_2 [ \hf^2_{nn_3}(n_A) \ff^2_{nn_3}(t_A,n_A)  ], 
\label{RGG1}
\end{align}

\noi
where
$I_2$ denotes the multiple Wiener  integral (see Subsection \ref{SUBSEC:FBM}).
Here,  
with  $A = \{1, 2\}$, 
$\hf^2_{nn_3}(n_A) $ and $\ff^2_{nn_3}(t_A, n_A) $ are defined by 
\begin{align}
\begin{split}
\hf^2_{n n_3}(n_A) & = 
 \ind_{n=n_{123}}
\cdot 
\frac {\jb{n}^{s_0 }} {\jb{n_3}^s  }\phi_{n_1}
\phi_{n_2}, \\
 \ff^2_{nn_3} (t_A, n_A)  &= 
\ff_{nn_3}^{2, t, r} (t_A, n_A) = 
\ind_{r \le t_2 \le t_1\le t}\\
& \quad \times 
 e^{-(t-t_1)\jb{n}^2}
e^{-(t_1-t_2)\jb{n_{23}}^2}
e^{- (t_2-r)\jb{n_3}^2}.
\end{split}
\label{RGG2}
\end{align}

\noi
Given $\Nb_3 \in (2^{\Z_{\ge 0}})^5$
of the form \eqref{XY1}, we set 
\begin{align}
\begin{split}
 \hf^{2, \Nb_3}_{n n_3}(n_A) 
% & = 
%\hf^{2, N, N_1, N_2, N_3, N_{23}}_{n n_3}(n_A)
& = \ind_{E_{\Nb_3}}
\cdot 
 \hf^2_{n n_3}(n_A), \\
\ff^{2, \bf N}_{nn_3} (t_A, n_A)  
%&= 
% \ff^{2, N, N_1, N_2, N_3, N_{23}}_{nn_3} (t_A, n_A)  = 
& = \ind_{E_{\Nb_3}}
\cdot 
 \ff^2_{nn_3} (t_A, n_A),  
\end{split}
\label{RGG3}
\end{align}

\noi
where $ E_{\Nb_3}$ is as in \eqref{XY2}.
Then, given any $\ta > 0$, 
it follows from \eqref{RGG1}, \eqref{RGG2}, \eqref{RGG3}, 
and the random tensor estimate (Lemma~\ref{LEM:RT})
that 
\begin{align}
\begin{split}
 \big\| &  \| \XX^2_{t,r} \|_{\LOP(H^s;H^{s_0})} \big\|_{L^p(\Om)}
 = \big\| \| \jb{\nb}^{s_0} \XX^2_{t,r}\jb{\nb}^{-s} \|_{\LOP(L^2;L^2)} \big\|_{L^p(\Om)}\\
& = \big\| \| I_2[  \hf^2_{nn_3}(n_A)  \ff^2_{nn_3}(t_A, n_A)  ] \|_{\l^2_{n_3} \to \l^2_n} \big\|_{L^p(\Om)}\\
& \le \sum_{\Nb_3\in \D_3}
\big\| \| 
I_2[ \hf^{2, \Nb_3}_{nn_3}(n_A)  \ff^{2, \Nb_3}_{nn_3}(t_A, n_A)] \|_{\l^2_{n_3} \to \l^2_n} \big\|_{L^p(\Om)}\\
& \les 
p
\sum_{\Nb_3\in \D_3}
\frac{N_{\max}^{\frac 12 \ta} N^{\ta s_0}}
{N_1^{ \ta\s} N_2^{\ta \s} N_3^{\ta s} }
\|\phi\|_{\HS(L^2; H^\s)}^{2\ta}\\
& \hphantom{XXXXX}\times 
 \|  \ff^{2, \Nb_3}_{nn_3} (t_A, n_A)\|_{\l^\infty_{nn_3n_A}\Hs^\be_{t_A} }
 \Big(\max_{(B, C)} \|  \hf^{2, \Nb_3}_{nn_3}(n_A)\|_{n_3n_B \to nn_C} \Big)^{1-\ta}
\end{split}
\label{RGG4}
\end{align}

\noi
for any finite $p \ge 1$, 
where $\D_3$ is as in \eqref{XY3a}, 
 $N_{\max} = \max(N, N_1, N_2, N_3)$,
  $\Hs^\be_{t_A}(\R^2_+)$ is as in~\eqref{BM0a},  
  and 
  the maximum 
on the last factor in \eqref{RGG4}
is taken over all  partitions
$(B, C)$  
 of $A = \{1, 2\}$:
\[n_1 n_2 n_3  \to n , \qquad  
n_1 n_3 \to n n_2, 
\qquad n_2 n_3 \to nn_1, 
\qquad \text{and}\qquad 
n_3 \to nn_1 n_2. \]

From \eqref{RGG2}, \eqref{RGG3}, and 
Cauchy-Schwarz's inequality with \eqref{phi1a}, we have
\begin{align}
\begin{split}
& \|   \hf^{2, \Nb_3}_{nn_3}(n_A)\|_{n_1n_2n_3 \to n} \\
&\quad \sim 
\frac{N^{s_0}}{N_3^s  }
\sup_{\|f\|_{\l^2_{n_1 n_2 n_3}}
= \|g\|_{\l^2_{n}}  = 1}
\bigg|
\sum_{n, n_1, n_2, n_3 \in \Z}
\ind_{ E_{\Nb_3}}
\phi_{n_1}\phi_{n_2} f_{n_1 n_2 n_3} g_{n}\bigg|\\
&\quad \les
\frac{N^{s_0}}{N_1^\s N_2^\s N_3^s}
\|\phi\|_{\HS(L^2; H^\s)}^2.
\end{split}
\label{RGG4a}
\end{align}

\noi
A similar computation yields
\begin{align}
\max_{(B, C)} \|  \hf^{2, \Nb_3}_{nn_3}(n_A)\|_{n_3n_B \to nn_C} 
&\les
\frac{N^{s_0}}{N_1^\s N_2^\s N_3^s}
\|\phi\|_{\HS(L^2; H^\s)}^2.
\label{RGG6}
\end{align}

From  \eqref{BM0a} and Sobolev's inequality
(separately applied
to each of the $t_1$- and $t_2$- variables), we have  
\begin{align}
\begin{split}
 \|  \ff^2_{nn_3} (t_A, n_A)\|_{\l^\infty_{nn_3n_A} \Hs_{t_A}^\be}
& =  
\|  \ff^2_{nn_3} (t_A, n_A)\|_{\l^\infty_{nn_3n_A} \dot H_{t_A}^{\frac 12 - \be}}\\
& \les 
\|  \ff^2_{nn_3} (t_A, n_A)\|_{\l^\infty_{nn_3n_A}  L^\frac 1\be_{t_A}}.
\end{split}
\label{ZYS1}
\end{align}

\noi
By applying the bound on $J^1$ in \eqref{YS3}
and another change of variables with \eqref{YS2}, we have 
\begin{align}
\begin{split}
\| &  \ff^2_{nn_3} (t_A, n_A)\|_{ L^\frac 1\be_{t_A} }^\frac 1 \be 
 = J^2_{t, r}(n, n_{23}, n_3)\\
: \!& = \int_r^t 
e^{-\be^{-1}(t-t_1)\jb{n}^2}
\int_r^{t_1}
e^{-\be^{-1}(t_1-t_2)\jb{n_{23}}^2}
e^{- \be^{-1}(t_2-r)\jb{n_3}^2}dt_2  
dt_1 \\
& = \int_r^t 
e^{-\be^{-1}(t-t_1)\jb{n}^2}
J^1_{t_1, r}(n_{23}, n_3)
%\int_r^{t_1}
%e^{-\be^{-1}(t_1-t_2)\jb{n_{23}}^2}
%e^{- \be^{-1}(t_2-r)\jb{n_3}^2}dt_2  
dt_1 \\
& \les_{\be, a_1, a_2, a_3} 
\jb{n}^{-2a_1}\jb{n_{23}}^{-2a_2}
\jb{n_3}^{-2a_3}\\
& \quad \times \int_r^t (t - t_1)^{-a_1} 
(t_1- r)^{1-a_2- a_3}  dt_1\\
& 
\sim_{a_1, a_2, a_3} 
 (t-r)^{2-\sum_{j = 1}^3 a_j }
\jb{n}^{-2a_1}\jb{n_{23}}^{-2a_2}
\jb{n_3}^{-2a_3}.
\end{split}
\label{ZYS3}
\end{align}

\noi
uniformly in $n, n_2, n_3 \in \Z$
and  $t >  r\ge 0$, 
provided that $0 \le a_1, a_2, a_3 < 1$, 
satisfying~\eqref{XY5}.
Hence,  from \eqref{ZYS1} and \eqref{ZYS3}, we have
\begin{align}
\begin{split}
\|  \ff^{2, \Nb_3}_{nn_3} (t_A, n_A)\|_{ L^\frac 1\be_{t_A} }
& \les_{a_1, a_2, a_3} (t-r)^{(2
-\sum_{j = 1}^3 a_j )\be}N^{-2a_1\be}
N_{23}^{-2a_2\be}N_{3}^{-2a_3\be}, 
\end{split}
\label{ZYS3a}
\end{align}

\noi
uniformly in $(n, n_1, n_2, n_3) \in E_{\Nb_3}$, $n_A \in \Z^A$, 
$\Nb_3 \in \D_3$, 
and $t >  r\ge 0$.

Therefore, putting 
\eqref{RGG4}, 
\eqref{RGG6}, 
and \eqref{ZYS3a}
together with \eqref{XY6}, 
we obtain the desired bound \eqref{XY7}.
\end{proof}

\medskip

\noi
$\bul$ {\bf Third order driver $\XX^3$ in \eqref{high3a}.}\rule[-3mm]{0pt}{0pt}\\
\indent
Given dyadic $N, N_1, N_2, N_3, N_4, N_{234}, N_{23} \ge1$, 
we use $\Nb_4$ to denote the following dyadic septuple:
\begin{align}
\Nb_4 = (N, N_1, N_2, N_3, N_4, N_{234}, N_{34}) \in (2^{\Z_{\ge 0}})^7.
\label{XXY1}
\end{align}

\noi
Given $\Nb_4 \in (2^{\Z_{\ge 0}})^7$, 
define  $ E_{\Nb_4}$ by 
\begin{align}
\begin{split}
E_{\Nb_4} = \big\{(n, n_1, n_2, n_3, n_4) \in \Z^5:
\ & n = n_{1234}, \,  |n|\sim N, \\
&  |n_j|\sim N_j, \, j = 1, \dots, 4, \\
&  |n_{234}|\sim N_{234},
\, |n_{34}|\sim N_{34}
\big\}.
\end{split}
\label{XXY2}
\end{align}

\noi
On $E_{\Nb_4}$, we have
\begin{align}
\max( N, N_{234}, N_{34}, N_4) \sim N_{\max} : = \max (N, N_1, N_2, N_3, N_4).
\label{XXY3}
\end{align}

\noi
We define $\D_4$ by setting
\begin{align}
\D_4 = \big\{
\Nb_4 \in (2^{\Z_{\ge 0}})^7: \, \text{$\Nb_4$ is  of the form \eqref{XXY1}, satisfying \eqref{XXY3}}
\big\}.
\label{XXY3a}
\end{align}

%
%note by $\D_4$ the collection  of 
%dyadic septuples
%$\Nb_4$ of the form \eqref{XXY1}, satisfying \eqref{XXY2}.
\noi
We also denote by  $\A_4$ the class of functions
$\ab = (a_1, a_2, a_3, a_4)
:(2^{\Z_{\ge 0}})^7 \to [0, 1)^4$
of the form: 
\begin{align}
 \ab(\Nb_4) = \big(a_1(\Nb_4), a_2(\Nb_4), a_3(\Nb_4), a_4(\Nb_4)\big) \in [0, 1)^4
 \label{XXY4} 
\end{align}

\noi
for $\Nb_4 \in (2^{\Z_\ge 0})^7$, 
satisfying
\begin{align}
\sum_{j = 1}^4 a_j \le 3.
\label{XXY5}
\end{align}

\begin{proposition}
\label{PROP:drive4}
Let  $\frac 12 \le \be < 1$
and $s, s_0, \s \in \R$.
Given $\phi \in \HS(L^2(\T); H^\s(\T))$
satisfying~\eqref{phi1}, 
let  $\XX^3$ be the third order driver in~\eqref{high3a}.
Given
small $\ta > 0$ and 
$\ab \in \A_4$, 
set 
\begin{align}
M^{(3)}_{\be, \s, s, s_0, \ta} (\ab(\Nb_4))
= 
\frac{N_{\max}^{\frac 12 \ta} N^{s_0- 2a_1(\Nb_4)\be}}
{N_1^{\s} N_2^\s N_3^\s N_{234}^{2a_2(\Nb_4)\be }
N_{34}^{2a_3(\Nb_4)\be }
N_4^{s+2a_4(\Nb_4)\be}}
\label{XXY6}
\end{align}

\noi
for $\Nb_4$ of the form \eqref{XXY1}, 
where $N_{\max} = \max(N, N_1, \dots, N_4)$.
 Then, given any finite $p \ge 1$ and $\ab \in \A_4$, we have
\begin{align}
\begin{split}
 & \big\| \| \XX^3_{t,r}  \|_{\LOP(H^{s}; H^{s_0})} 
  \big\|_{L^p(\Om)} 
   \les_\ab    p^\frac 32  \| \phi \|_{\HS(L^2;H^\s)}^3  \\
& \hphantom{XXXXXX}\times 
\sum_{\Nb_4 \in \D_4}
M^{(3)}_{\be, \s, s, s_0, \ta} (\ab(\Nb_4))
 (t-r)^{(3
-\sum_{j = 1}^4 a_j (\Nb_4))\be}
\end{split}
\label{XXY7}
\end{align}

\noi
for any   $t>r\ge 0$, uniformly in $\ab \in \A_4$.

Furthermore, suppose that the following holds
for a given choice of  $\ab \in \A_4$\textup{:}
\smallskip

\begin{itemize}
\item[(i)]
$M^{(3)}_{\be, \s, s, s_0, \ta} (\ab(\Nb_4))$ in \eqref{XXY6} is summable\textup{:}
\begin{align}
\sum_{\Nb_4 \in \D_4}
M^{(3)}_{\be, \s, s, s_0, \ta} (\ab(\Nb_4)) < \infty, 
\label{XXY8}
\end{align}

\smallskip

\item[(ii)] 
there exists $0 < \g < 1$ such that 
\begin{align}
\inf_{\Nb_4\in \D_4}\bigg(3
-\sum_{j = 1}^4 a_j (\Nb_4)\bigg)\be >  \g, 
\label{XXY9}
\end{align}

\smallskip

\item [(iii)]
for the choice of $\g$ in (ii), the bounds \eqref{YG1a}
and \eqref{XY10} 
for $\XX^1$ and $\XX^2$, respectively,  hold.

\end{itemize}

\smallskip

\noi
Then, 
 we have 
\begin{align}
\big\| \| \XX^3 \|_{C^{\g}_{2,T} \LOP(H^s;H^{s_0})} \big\|_{L^p(\Om)} \les_T p^\frac 32   \| \phi \|_{{\HS} (L^2; H^\s)}^3
\label{XXY10}
\end{align}

\noi
for 
any finite $p \ge 1$ and $T > 0$.
In particular, 
there exists a version of $\XX^3$
such that 
\[\XX^3 \in C^{\g}_{2,T} \LOP(H^s(\T);H^{s_0}(\T)),\]

\noi
almost surely.

\end{proposition}

In the next subsection, 
we make a suitable choice of $\ab \in \A_4$, 
guaranteeing \eqref{XXY8}
and~\eqref{XXY9}.

\begin{proof}[Proof of Proposition \ref{PROP:drive4}]
As in the  proof of Proposition \ref{PROP:drive3}, 
it suffices to prove 
 the bound~\eqref{XXY7},
since the bound \eqref{XXY10} follows
from~\eqref{XXY7}, \eqref{XXY8}, \eqref{XXY9},  and 
 \cite[Lemma~3.8]{GT10} with \eqref{chen2} and 
 the bounds \eqref{YG1a}
 and \eqref{XY10}
 for $\XX^1$ and $\XX^2$.

Fix $\frac 12 \le \be < 1$
and $t  > r \ge 0$.
From \eqref{high3a}, we have 
\begin{align}
\Ft_x\big( \jb{\nabla}^{s_0} \XX^3_{t,r} \jb{\nabla}^{-s} f\big) (n) 
= \sum_{n_4 \in \Z} \ft f(n_4) I_3 [ \hf^3_{nn_4}(n_A) \ff^3_{nn_4}(t_A,n_A)  ], 
\label{ZRGG1}
\end{align}

\noi
where
$I_3$ denotes the multiple Wiener integral (see Subsection \ref{SUBSEC:FBM}).
Here,  
with  $A = \{1, 2, 3\}$, 
$\hf^3_{nn_4}(n_A) $ and $\ff^3_{nn_4}(t_A, n_A)$ are defined by 
\begin{align}
\begin{split}
\hf^3_{n n_4}(n_A) & = 
 \ind_{n=n_{1234}}
\cdot 
\frac {\jb{n}^{s_0 }} {\jb{n_4}^s  }
\prod_{j = 1}^3\phi_{n_j} , \\
 \ff^3_{nn_4} (t_A, n_A)  &= 
\ff_{nn_4}^{3, t, r} (t_A, n_A) = 
\ind_{r \le t_3 \le t_2 \le t_1\le t}\\
& \quad \times 
 e^{-(t-t_1)\jb{n}^2}
e^{-(t_1-t_2)\jb{n_{234}}^2}
e^{- (t_2-t_3)\jb{n_{34}}^2}
e^{- (t_3-r)\jb{n_4}^2}.
\end{split}
\label{ZRGG2}
\end{align}

\noi
Given $\Nb_4 \in (2^{\Z_{\ge 0}})^7$
of the form \eqref{XXY1}, we set 
\begin{align}
\begin{split}
 \hf^{3, \Nb_4}_{n n_4}(n_A) 
 & = \ind_{E_{\Nb_4}}
\cdot 
 \hf^3_{n n_4}(n_A), \\
\ff^{3, \Nb_4}_{nn_4} (t_A, n_A) 
 &= \ind_{E_{\Nb_4}}
\cdot 
 \ff^3_{nn_4} (t_A, n_A),  
\end{split}
\label{ZRGG3}
\end{align}

\noi
where $ E_{\Nb_4}$ is as in \eqref{XXY2}.
Then, given any $\ta > 0$, 
it follows from \eqref{ZRGG1}, \eqref{ZRGG2}, \eqref{ZRGG3}, 
and the random tensor estimate (Lemma~\ref{LEM:RT})
that 
\begin{align}
\begin{split}
 \big\| &  \| \XX^3_{t,r} \|_{\LOP(H^s;H^{s_0})} \big\|_{L^p(\Om)}
 = \big\| \| \jb{\nb}^{s_0} \XX^3_{t,r}\jb{\nb}^{-s} \|_{\LOP(L^2;L^2)} \big\|_{L^p(\Om)}\\
& = \big\| \| I_3[  \hf^3_{nn_4}(n_A)  \ff^3_{nn_4}(t_A, n_A)  ] \|_{\l^2_{n_4} \to \l^2_n} \big\|_{L^p(\Om)}\\
& \le \sum_{\Nb_4 \in \D_4}
\big\| \| 
I_3[ \hf^{3, \Nb_4}_{nn_4}(n_A)  \ff^{3, \Nb_4}_{nn_4}(t_A, n_A)] \|_{\l^2_{n_4} \to \l^2_n} \big\|_{L^p(\Om)}\\
& \les 
p^\frac 32
\sum_{\Nb_4 \in \D_4}
\frac{N_{\max}^{\frac 12 \ta} N^{\ta s_0}}
{N_1^{ \ta\s} N_2^{\ta \s} N_3^{\ta \s} N_4^{\ta s} }
\|\phi\|_{\HS(L^2; H^\s)}^{3\ta}\\
& \hphantom{XXXXX}\times 
 \|  \ff^{3, \Nb_4}_{nn_4} (t_A, n_A)\|_{\l^\infty_{nn_4n_A}\Hs^\be_{t_A} }
 \Big(\max_{(B, C)} \|  \hf^{3, \Nb_4}_{nn_4}(n_A)\|_{n_4n_B \to nn_C} \Big)^{1-\ta}
\end{split}
\label{ZRGG4}
\end{align}

\noi
for any finite $p \ge 1$, 
where $\D_4$ is as in \eqref{XXY3a}, 
 $N_{\max} = \max(N, N_1, \dots, N_4)$,
  $\Hs^\be_{t_A}(\R^3_+)$ is as in~\eqref{BM0a},  
  and the maximum on the last factor is taken 
  over all partitions 
$(B, C)$ of $A = \{1, 2, 3\}$.

Proceeding as in \eqref{YG3x} and \eqref{RGG4a} with Cauchy-Schwarz's inequality, 
we have 
\begin{align}
\max_{(B, C)} \|  \hf^{3, \Nb_4}_{nn_4}(n_A)\|_{n_4n_B \to nn_C} 
&\les
\frac{N^{s_0}}{N_1^\s N_2^\s  N_3^\s N_4^s}
\|\phi\|_{\HS(L^2; H^\s)}^3.
\label{ZRGG6}
\end{align}

\noi
By applying Sobolev's inequality 
(separately applied to each of the $t_1$-, $t_2$-, and $t_3$-variables)
with \eqref{BM0a}, the bound on $J^2$ in \eqref{ZYS3}, 
and another change of variables,  we have 
\begin{align}
%\begin{split}
  \| &  \ff^3_{nn_4} (t_A, n_A)\|_{ \Hs_{t_A}^\be}^\frac 1 \be 
 \les \|  \ff^3_{nn_4} (t_A, n_A)\|_{ L^\frac 1\be_{t_A} }^\frac 1 \be 
 \notag \\
& = \int_r^t 
\int_r^{t_1}
\int_r^{t_2}
 e^{-\be^{-1}(t-t_1)\jb{n}^2}
e^{-\be^{-1}(t_1-t_2)\jb{n_{234}}^2}
 \notag \\
& \hphantom{XXX}
\times e^{- \be^{-1}(t_2-t_3)\jb{n_{34}}^2}
e^{- \be^{-1}(t_3-r)\jb{n_4}^2}
dt_3 dt_2  
dt_1 
 \notag \\
& = \int_r^t 
e^{-\be^{-1}(t-t_1)\jb{n}^2}
J^2_{t_1, r}(n_{234},n_{34},  n_4)
dt_1 
\label{ZZYS3}\\
& \les_{\be, a_1, \dots, a_4} 
\jb{n}^{-2a_1}\jb{n_{234}}^{-2a_2}
\jb{n_{34}}^{-2a_3}
\jb{n_4}^{-2a_4}
 \notag \\
& \quad \times \int_r^t (t - t_1)^{-a_1} 
(t_1- r)^{2-\sum_{j =2}^4a_j}  dt_1
 \notag \\
& 
\sim_{a_1, \dots, a_4} 
 (t-r)^{3-\sum_{j = 1}^4 a_j }
\jb{n}^{-2a_1}\jb{n_{234}}^{-2a_2}
\jb{n_{34}}^{-2a_3}
\jb{n_4}^{-2a_4}, 
 \notag 
%\end{split}
\end{align}

\noi
\noi
uniformly in $n, n_2, n_3, n_4 \in \Z$
and  $t >  r\ge 0$, 
provided that $0 \le a_1, \dots , a_4 < 1$, 
satisfying~\eqref{XXY5}.
Hence, we have 
\begin{align}
\begin{split}
& \|  \ff^{3, \Nb_4}_{nn_4} (t_A, n_A)\|_{ L^\frac 1\be_{t_A} }\\
& \quad \les_{a_1, \dots, a_4}  
 (t-r)^{(3-\sum_{j = 1}^4 a_j )\be }
N^{-2a_1\be }
N_{234}^{-2a_2\be }
N_{34}^{-2a_3\be }
N_{4}^{-2a_4\be}, 
\end{split}
\label{ZZYS3a}
\end{align}

\noi
uniformly in $(n, n_1, n_2, n_3, n_4) \in E_{\Nb_4}$, $n_A \in \Z^A$, 
$\Nb_4 \in \D_4$, 
and $t >  r\ge 0$.

Therefore, putting 
\eqref{ZRGG4}, 
\eqref{ZRGG6}, 
and \eqref{ZZYS3a}
together with \eqref{XXY6}, 
we obtain the desired bound \eqref{XXY7}.
\end{proof}

\subsection{Proof of Theorem \ref{THM:2}}
\label{SUBSEC:R3}

We conclude this section by presenting a proof of 
Theorem~\ref{THM:2}.
More precisely, we prove that 
under a suitable choice of parameters, 
the convolution RDE~\eqref{ro1a}
is pathwise globally well-posed.

\begin{proof}[Proof of Theorem \ref{THM:2}]

Let  $\be = \frac 12$ and fix $- \frac 12 < \s< 0$.
Given small $\eps > 0$ (to be chosen later), 
we set 
\begin{align}
 \al = \frac 14 + \eps\qquad \text{and}\qquad
\g = \frac 14 + 2\eps.
\label{ZZ0}
\end{align}

\medskip

\noi
\underline{\bf Part~1:}
We  show that there exists $ s> 0$ such that 
\eqref{ro2a} holds for any $T > 0$, almost surely

We first consider the first order driver $\XX^1$
in \eqref{sto1x}.
By setting $s_0 = s > 0$
and $a = \frac 12 - 6\eps$, 
the conditions in \eqref{YS0a}
of 
Proposition \ref{PROP:drive2} reduce to 
\begin{align}
\s > s - \frac 12 + 6\eps, 
\label{ZX1}
\end{align}

\noi
which guarantees that 
\begin{align}
\XX^1 \in \bigcap_{k = 1}^\infty C^{\g}_{2,k} \LOP(H^s(\T)),
\label{ZZX1}
\end{align}

\noi
almost surely; see \eqref{YG1} and \eqref{YG1a}.

Next, we consider the second order driver $\XX^2$
in \eqref{high2}.
Given $\Nb_3 \in \D_3$ (see~\eqref{XY3a}), 
we set 
\begin{align}
a_{j_*}(\Nb_3) = 1 - 10 \eps
\label{ZX1a}
\end{align}

\noi
 for $j_* \in \{1, 2, 3\}$ corresponding to 
the maximum of $N$, $N_{23}$, and $N_3$ in \eqref{XY6}, 
and set $a_j(\Nb_3)  = 0$ for $j \ne j_*$.
Then, in  view of~\eqref{XY3} (recall $\be = \frac 12$), 
we have 
\begin{align}
 N^{-a_1(\Nb_3)}
 N_{23}^{- a_2(\Nb_3) }N_3^{- a_3(\Nb_3)}
\les N_{\max}^{-1 + 10 \eps}.
\label{ZX2}
\end{align}

\noi
Then, from \eqref{XY6} (with $s_0 = s$) and \eqref{ZX2}, we have 
\begin{align*}
\sum_{\Nb_3 \in \D_3} M^{(2)}_{\be, \s, s, s, \ta} (\ab(\Nb_3))
= 
\sum_{\Nb_3 \in \D_3} 
\frac{N_{\max}^{\frac 12 \ta - 1 + 10 \eps} N^{s}}
{N_1^{\s} N_2^\s N_3^{s}} < \infty, 
%\label{ZX3}
\end{align*}

\noi
provided that 
\begin{align}
\s > \frac 12 s - \frac 12 + \frac 14 \ta + 5 \eps.
\label{ZX4}
\end{align}

\noi
As a result,
it follows from Proposition \ref{PROP:drive3} with \eqref{ZX1a} and \eqref{ZZ0} that, 
 given any finite $p \ge 1$, we have
\begin{align*}
\big\| \| \XX^2_{t,r}  \|_{\LOP(H^{s})}
  \big\|_{L^p(\Om)} 
   \les   p  \| \phi \|_{\HS(L^2;H^\s)}^2  
 (t-r)^{2\g + \eps}
%\label{ZX5}
\end{align*}

\noi
for any   $t>r\ge 0$
and thus 
\begin{align}
\XX^2 \in \bigcap_{k = 1}^\infty C^{2\g}_{2,k} \LOP(H^s(\T)),
\label{ZZX2}
\end{align}

\noi
almost surely; see \eqref{XY7} and \eqref{XY10}.

We now consider the third order driver $\XX^3$
in \eqref{high3}.
Given $\Nb_4 \in \D_4$ (see \eqref{XXY3a})
we choose $\ab(\Nb_4)$ in \eqref{XXY4} such that 
\begin{align}
 \sum_{j = 1}^4 a_j(\Nb_4) = \frac 32 - 14\eps.
\label{ZX6}
\end{align}

\medskip

\noi
$\bul$ {\bf Case 1:}
$N \vee N_{234} \sim N_{\max}$.\\
\indent
In this case, we set $a_j(\Nb_4)  = 1- 13\eps$
for $j = 1$ if $N \ge N_{234}$
and $j = 2$ if $N < N_{234}$.
On $E_{\Nb_4}$ defined in \eqref{XXY2}, we have
$N_{34} \vee N_4 \ges N_{3}$.
Then, 
we set $a_j(\Nb_4)  = \frac 12 - \eps$
for $j = 3$ if $N_{34} \ge N_{4}$
and $j = 4$ if $N_{34} < N_{4}$.
We set $a_j(\Nb_4) = 0$ for the remaining $j$'s.
Then, we have 
\begin{align}
\frac{N_{\max}^{\frac 12 \ta} N^{- a_1(\Nb_4)}}
{N_3^\s N_{234}^{a_2(\Nb_4) }
N_{34}^{a_3(\Nb_4) }
N_4^{a_4(\Nb_4)}}
\les N_{\max}^{\frac 12 \ta - 1 + 13\eps}, 
\label{ZX7}
\end{align}

\noi
provided that $\eps > 0$ is sufficiently small such that 
\begin{align}
\s \ge - \frac 12 + \eps.
\label{ZX8}
\end{align}

\medskip

\noi
$\bul$ {\bf Case 2:}
$N,  N_{234} \ll  N_{\max}$.\\
\indent 
In this case, we have $N_{34} \vee N_4 \sim N_{\max}$.
We set $a_j(\Nb_4)  = 1- 13\eps$
for $j = 3$ if $N_{34} \ge N_{4}$
and $j = 4$ if $N_{34} < N_{4}$.
On $E_{\Nb_4}$ defined in \eqref{XXY2}, we have
$N \vee N_{234} \ges N_{1}$.
Then, 
we set $a_j(\Nb_4)  = \frac 12 - \eps$
for $j = 1$ if $N \ge N_{234}$
and $j = 2$ if $N < N_{234}$.
We set $a_j(\Nb_4) = 0$ for the remaining $j$'s.
Then, we have 
\begin{align}
\frac{N_{\max}^{\frac 12 \ta} N^{- a_1(\Nb_4)}}
{N_1^\s N_{234}^{a_2(\Nb_4) }
N_{34}^{a_3(\Nb_4) }
N_4^{a_4(\Nb_4)}}
\les N_{\max}^{\frac 12 \ta - 1 + 13\eps}, 
\label{ZX9}
\end{align}

\noi
provided that \eqref{ZX8} holds.

\medskip

From \eqref{XXY6}, \eqref{ZX7},  and \eqref{ZX9}, we have 
\begin{align*}
\sum_{\Nb_4 \in \D_4} M^{(3)}_{\be, \s, s, s, \ta} (\ab(\Nb_4))
 < \infty, 
%\label{ZX10}
\end{align*}

\noi
provided that 
\begin{align}
\s > \frac 12 s - \frac 12 + \frac 14 \ta + \frac{13} 2 \eps.
\label{ZX11}
\end{align}

\noi
As a result,
it follows from Proposition \ref{PROP:drive4} with \eqref{ZX6}  
and \eqref{ZZ0} that, 
 given any finite $p \ge 1$, we have
\begin{align*}
\begin{split}
\big\| \| \XX^3_{t,r}  \|_{\LOP(H^{s})}
  \big\|_{L^p(\Om)} 
   \les   p^\frac 32  \| \phi \|_{\HS(L^2;H^\s)}^3
 (t-r)^{3\g + \eps}
\end{split}
%\label{ZX12}
\end{align*}

\noi
for any   $t>r\ge 0$
and thus 
\begin{align}
\XX^3 \in \bigcap_{k = 1}^\infty C^{3\g}_{2,k} \LOP(H^s(\T)),
\label{ZZX3}
\end{align}

\noi
almost surely; see \eqref{XXY7} and \eqref{XXY10}.

\medskip

\noi
$\bul$ {\bf Summary:}
Let  $- \frac 12 < \s< 0$.
Then, there exists small $s, \eps, \ta > 0$ such  that 
\eqref{ZX1}, \eqref{ZX4}, \eqref{ZX8}, and \eqref{ZX11}
hold.
Hence, 
it follows from \eqref{ZZX1},\eqref{ZZX2}, and \eqref{ZZX3}
that 
there exists $\Si \subset \Om$ with $\PP(\Si) = 1$ such that, 
for each $\o \in \Si$ and $j = 1,2, 3$, we have 
\begin{align*}
\XX^j=  \XX^j (\o)
\in \bigcap_{k = 1}^\infty C^{j\g }_{2,k} \L(H^s(\T)).
%\label{ZZ1}
\end{align*}

\noi
In the following, we fix $\o \in \Si$
and often suppress dependence on $\o$.

\medskip

\noi
\underline{\bf Part 2:}
In the following, we briefly discuss pathwise  well-posedness of the convolution RDE~\eqref{ro1a}.

Let $\g > \al > \frac 14$ be as in \eqref{ZZ0} and fix $0 < T \le 1$ (to be chosen later).
Given $u_0 \in H^s(\T)$, 
define $\D_{\XX^1, \XX^2, T}^{\al, \g, u_0}(H^s(\T))$
to be the subclass
of $\D_{\XX^1, \XX^2, T}^{\al, \g}(H^s(\T))$
defined in  Definition~\ref{DEF:high}\,(i)
such that $u|_{t = 0} = u^1|_{t = 0} = u^2|_{t = 0}  = u_0$.
Given $\uu = (u, u^1, u^2) 
\in \D_{\XX^1, \XX^2, T}^{\al, \g, u_0}(H^s(\T))$, satisfying
\begin{align}
\begin{split}
\ft \updl u & = \XX^1 (u^1) + \XX^2(u^2) + R^{0, \uu}, \\
\ft \updl u^1 & = \XX^1 (u^2)  + R^{1, \uu}
\end{split}
\label{ZZ1a}
\end{align}

\noi
for some 
 $R^{0, \uu} \in C_{2, T}^{\al + 2\g}H^s(\T)$
 and 
$R^{1, \uu} \in C_{2, T}^{\al + \g}H^s(\T)$
(with the short-hand notation \eqref{prod1}), 
define the process $z = z(u, u^1, u^2)$ by 
\begin{align}
\begin{split}
z(t) 
& = S(t) u_0 + \I^{\cj \XX}(u)(t)\\
& = S(t) u_0 + 
\Big((\Id - \ft \Ld \ft \updl)
 \big[\XX^1 (u) 
+    \XX^2 (u^1) 
+   \XX^3 (u^2) \big] \Big)_{t, 0}.
\end{split}
\label{ZZ2}
\end{align}

\noi
See \eqref{ro9}.
Then, from 
the convolution sewing lemma (Lemma \ref{LEM:sew2}) with 
\eqref{ro5a} and~\eqref{ro6} (recall that $\al + 3\g > 1$), we have 
\begin{align}
\begin{split}
\ft \updl z 
& = (\Id - \ft \Ld \ft \updl)
 \big[\XX^1 (u) 
+   \XX^2 (u^1) 
+   \XX^3 (u^2) \big] \\
& = 
\XX^1 (u) 
+   \XX^2 (u^1) + R^{0, \zz}, 
\end{split}
\label{ZZ3}
\end{align}

\noi
where $R^{0, \zz}$ is given by 
\begin{align}
\begin{split}
R^{0, \zz} 
& = 
\XX^3 (u^2) 
-  \ft \Ld \ft \updl
 \big[\XX^1 (u) 
+   \XX^2 (u^1) 
+   \XX^3 (u^2) \big] 
\\
& = 
\XX^3 (u^2) 
+ \ft \Ld \big[\XX^1(R^{0, \uu}) + 
\XX^2(R^{1, \uu})
+ \XX^3(\ft \updl u^2)\big]. % \in C^{\al + 2\g}_{2, T}H^s(\T).
\end{split}
\label{ZZ4}
\end{align}

\noi
Here, the second equality follows from \eqref{ro5a} and \eqref{ro6}.
By comparing  \eqref{ZZ3} with \eqref{control1}, we can take the first Gubinelli derivative
$z^1$ of $z$ to be\footnote{There is no a priori uniqueness
of the Gubinelli derivative under smooth perturbation.} 
\begin{align}
z^1 = u. %\qquad \text{and}\qquad z^2 = u^1.
\label{ZZ5}
\end{align}

\noi
Then, from \eqref{ZZ5} and \eqref{ZZ1a}, we have 
\begin{align}
\ft \updl z^1  = \XX^1 (u^1)  + R^{1, \zz}, 
\label{ZZ6}
\end{align}

\noi
where $R^{1, \zz}$ is given by 
\begin{align}
R^{1, \zz} 
=  \XX^2(u^2) + R^{0, \uu}.
% \in C^{\al + \g}_{2, T}H^s(\T).
\label{ZZ6a}
\end{align}

\noi
By comparing \eqref{ZZ6} with \eqref{control1}, 
we can take the second Gubinelli derivative $z^2$ of $z$ to be 
\begin{align}
z^2 = u^1.
\label{ZZ7}
\end{align}

\noi
Hence, we
define a map $\G$ on 
$\D_{\XX^1, \XX^2, T}^{\al, \g, u_0}(H^s(\T))$
%$\big(\ft \CC_{T}^{\al} H^s(\T)\big)^{\otimes 3}$ 
by setting
\begin{align}
\begin{split}
\G(\uu) = \G(u, u^1, u^2)
& = \zz = (z, z^1, z^2)\\
& = (z, u, u^1), 
\end{split}
\label{ZZ8}
\end{align}

\noi
where $z$ is as in \eqref{ZZ2} and \eqref{ZZ3}
and the last equality holds thanks to \eqref{ZZ5} and \eqref{ZZ7}.

From \eqref{ZZ5}, \eqref{Ho4}, \eqref{Ho5}, \eqref{Ho3}, 
and \eqref{ZZ1a}
with $z^1|_{t = 0} = u|_{t = 0} = u_0$, 
 we have
\begin{align}
\begin{split}
\|z^1\|_{\ft \CC^\al_T H^s_x}
& \les \| \ft \updl u \|_{C^\al_{2, T} H^s_x} + \|u_0\|_{H^s}\\
& \le T^{\g - \al}\|\XX^1\|_{C^\g_{2, T} \L(H^s)} \|u^1\|_{L^\infty_TH^s_x} \\
& \quad +
T^{2\g - \al}\|\XX^2\|_{C^{2\g}_{2, T} \L(H^s)} \|u^2\|_{L^\infty_TH^s_x}\\
& \quad + T^{2\g} \|R^{0, \uu}\|_{C^{\al+2\g}_{2, T} H^s_x}
+ \|u_0\|_{H^s}\\
& \le T^{\g - \al}\|\XX^1\|_{C^\g_{2, T} \L(H^s)} \|u^1\|_{\ft \CC_T^\al H^s_x}\\
& \quad +
T^{2\g - \al}\|\XX^2\|_{C^{2\g}_{2, T} \L(H^s)} \|u^2\|_{\ft \CC_T^\al H^s_x}\\
& \quad + T^{2\g} \|R^{0, \uu}\|_{C^{\al+2\g}_{2, T} H^s_x}
+ \|u_0\|_{H^s}.
\end{split}
\label{ZZ9}
\end{align}

\noi
Similarly, from   \eqref{ZZ7} and \eqref{ZZ1a}
with 
$z^2|_{t = 0} = u^1|_{t = 0} = u_0$, 
we have
\begin{align}
\begin{split}
\|z^2\|_{\ft \CC^\al_T H^s_x}
& \les \| \ft \updl u^1 \|_{C^\al_{2, T} H^s_x} + \|u_0\|_{H^s}\\
& \le T^{\g - \al}\|\XX^1\|_{C^\g_{2, T} \L(H^s)} \|u^2\|_{L^\infty_TH^s_x} \\
& \quad +
T^{\g} \|R^{1, \uu}\|_{C^{\al+\g}_{2, T} H^s_x}
+ \|u_0\|_{H^s}\\
& \le T^{\g - \al}
\|\XX^1\|_{C^\g_{2, T} \L(H^s)} \|u^2\|_{\ft \CC_T^\al H^s_x} \\
& \quad + T^{\g} \|R^{1, \uu}\|_{C^{\al+\g}_{2, T} H^s_x}
+ \|u_0\|_{H^s}.
\end{split}
\label{ZZ10}
\end{align}

\noi
From \eqref{ZZ4}
and the convolution sewing lemma (Lemma \ref{LEM:sew2}), we have 
\begin{align}
\begin{split}
\|R^{0, \zz}\|_{C^{\al + 2\g}_{2, T}H^s_x}
& \les
T^{\g - \al} \|\XX^3\|_{C^{3\g}_{2, T} \L(H^s)} \|u^2\|_{L^\infty_T H^s_x} \\
& \quad + 
T^\g \Big\|\ft \Ld \big[\XX^1(R^{0, \uu}) + 
\XX^2(R^{1, \uu})
+ \XX^3(\ft \updl u^2)\big] \Big\|_{C^{\al + 3\g}_{2, T}H^s_x}\\
& \le T^{\g - \al} \|\XX^3\|_{C^{3\g}_{2, T} \L(H^s)} \|u^2\|_{\ft \CC_T^\al H^s_x} \\
& \quad + T^\g \Big(\|\XX^1\|_{C^{\g}_{2, T} \L(H^s)}
\|R^{0, \uu}\|_{C^{\al+2\g}_{2, T} H^s_x}  \\
& \hphantom{lXXX}
+ \|\XX^2\|_{C^{2\g}_{2, T} \L(H^s)}
\|R^{1, \uu}\|_{C^{\al+\g}_{2, T} H^s_x}\\
& \hphantom{lXXX}
+ \|\XX^3\|_{C^{3\g}_{2, T} \L(H^s)}
\| u^2 \|_{\ft \CC_T^\al H^s_x}\Big).
\end{split}
\label{ZZ11}
\end{align}

\noi
From \eqref{ZZ6a}, 
we have  
\begin{align}
\begin{split}
\|R^{1, \zz}\|_{C^{\al + \g}_{2, T}H^s_x}
& \le  T^{\g-\al} 
\|\XX^2\|_{C^{2\g}_{2, T} \L(H^s)}
 \|u^2\|_{L^\infty_T H^s_x} 
+ T^\g \|R^{0, \uu}\|_{C^{\al + 2\g}_{2, T}H^s_x}\\
& \le  T^{\g-\al} 
\|\XX^2\|_{C^{2\g}_{2, T} \L(H^s)}
\|u^2\|_{\ft \CC_T^\al H^s_x} 
+ T^\g \|R^{0, \uu}\|_{C^{\al + 2\g}_{2, T}H^s_x}.
\end{split}
\label{ZZ12}
\end{align}

Therefore, 
putting \eqref{ZZ9}, \eqref{ZZ10}, 
\eqref{ZZ11}, and \eqref{ZZ12} together
with analogous difference estimates,   
we see that the map $\G$ defined in \eqref{ZZ8}
is a contraction on 
$ \D_{\XX^1, \XX^2, T}^{\al, \g, u_0}(H^s(\T))$
 (see~\eqref{control2}), 
by choosing
\begin{align}
T = T\Big(  \|\XX^1(\o) \|_{C^{\g}_{2, 1} \L(H^s)},
 \|\XX^2(\o) \|_{C^{2\g}_{2, 1} \L(H^s)},
  \|\XX^3(\o) \|_{C^{3\g}_{2, 1} \L(H^s)}
 \Big) >0
\label{ZZ13a}
\end{align}

\noi
sufficiently small.
Since a required argument is standard, we omit details.
See, for example, \cite[Subsection 4.4]{CGLLO2}
for details of such an argument 
(in the context of a nonlinear YDE with a rough perturbation).

By the Banach fixed point theorem, 
there exists a unique fixed point $\uu = (u, u^1, u^2) \in 
 \D_{\XX^1, \XX^2, T}^{\al, \g, u_0}(H^s(\T))$
for $\G$. In particular, from  \eqref{ZZ8}, 
we have 
\begin{align}
z = u = u^1 = u^2, 
\label{ZZ13}
\end{align}

\noi
and hence it follows
from \eqref{ZZ2} that 
$u$ is a unique solution to the convolution RDE~\eqref{ro1a}.
This proves pathwise  local well-posedness.

Pathwise 
global well-posedness follows as in the Young case by noting that 
the local existence time in \eqref{ZZ13a} does not depend
on the initial data, allowing us to iterate the local-in-time argument.
We omit details.
\end{proof}

\begin{ackno}\rm
 The authors would 
also like to express their gratitude to the anonymous referees for the helpful comments which improved the quality of the paper.
T.O.~was supported by the European Research Council (grant no.~864138 ``SingStochDispDyn")
and also 
acknowledges support from  
the NSFC (grant no.~W2531005).
Y.S.~is grateful to the financial supports of the National Key R$\&$D Program of China 
(grant no.~2022YFA1006300) and the NSFC (grant no.~12426205, no.~12271030).

\end{ackno}

\end{document}